\documentclass[12pt]{amsart}
\usepackage{color}

\usepackage{enumerate}
\usepackage{graphics}

\usepackage{mathrsfs}
\usepackage{amssymb}

\usepackage[all]{xy}
\usepackage{graphicx}
\usepackage{tikz}  

\usepackage{amsmath}
\usepackage{amssymb}
\usepackage{amsfonts}
\usepackage{mathrsfs}
\usepackage{mathtools}

\newcommand{\N}{\mathbb{N}}
\newcommand{\Z}{\mathbb{Z}}

\newcommand{\R}{\mathbb{R}}
\newcommand{\C}{\mathbb{C}}

\newcommand{\A}{\mathcal{A}}
\newcommand{\B}{\mathcal{B}}
\renewcommand{\S}{\mathcal{S}}
\newcommand{\SR}{\mathcal{S}(\R)}

\usepackage{dsfont}

\newcommand{\norm}[1]{\Vert#1\Vert}
\newcommand{\bignorm}[1]{\bigl\Vert#1\bigr\Vert}
\newcommand{\Bignorm}[1]{\Bigl\Vert#1\Bigr\Vert}

\newcommand{\norme}[1]{\left\Vert #1\right\Vert}
\newcommand{\normeinf}[1]{\norme{#1}_{\infty}}

\newcommand{\PM}{\mathcal{PM}}
\newcommand{\MH}{\mathcal{M}(H^1(\R))}
\newcommand{\MHp}{\mathcal{M}(H^p(\R))}

\theoremstyle{plain}
\newtheorem{thm}{Theorem}[section]
\newtheorem{prop}[thm]{Proposition}

\newtheorem{lem}[thm]{Lemma}

\newtheorem{cor}[thm]{Corollary}

\theoremstyle{definition}
\newtheorem{df}[thm]{Definition}
\theoremstyle{remark}

\newtheorem{rq1}[thm]{Remark}

\newtheorem{ex}[thm]{Example}

\numberwithin{equation}{section}

\theoremstyle{break}

\theoremstyle{nonumberplain}

\begin{document}

\title[Functional calculus on Hilbert space]
{Functional calculus for a bounded $C_0$-semigroup on Hilbert space}

\author[L. Arnold]{Loris Arnold}
\email{larnold@impan.pl}
\address{Institute of Mathematics, Polish Academy of Sciences, Śniadeckich 8, Warszawa, Poland}
\address{Universit\'e de Franche-Comt\'e, Laboratoire de Math\'ematiques de Besan\c con, UMR CNRS 6623,
16 Route de Gray, 25000 Besan\c{c}on, France}
\author[C. Le Merdy]{Christian Le Merdy}
\email{clemerdy@univ-fcomte.fr}
\address{Universit\'e de Franche-Comt\'e, Laboratoire de Math\'ematiques de Besan\c con, UMR CNRS 6623,
16 Route de Gray, 25000 Besan\c{c}on, France}

\date{\today}

\maketitle

\begin{abstract} 
We introduce a new Banach algebra $\A(\C_+)$
of bounded analytic  functions on $\C_+=\{z\in\C\, :\,
{\rm Re}(z)>0\}$ which is an analytic version
of the Figa-Talamenca-Herz algebras on $\R$. Then we
prove that the negative generator $A$ of any 
bounded $C_0$-semigroup on Hilbert space $H$
admits a bounded (natural) functional
calculus $\rho_A\colon \A(\C_+)\to B(H)$.
We prove that this is an improvement
of the bounded functional
calculus $\B_0(\C_+)\to B(H)$ recently devised by
Batty-Gomilko-Tomilov on a certain
Besov algebra $\B_0(\C_+)$ of analytic functions on $\C_+$,
by showing that $\B_0(\C_+)\subset \A(\C_+)$ and 
$\B_0(\C_+)\not= \A(\C_+)$. In the 
Banach space setting, we give similar results
for negative generators of
$\gamma$-bounded $C_0$-semigroups.
The study of $\A(\C_+)$ involves dealing with 
Fourier multipliers on the Hardy 
space $H^1(\R)\subset L^1(\R)$ of analytic functions.
\end{abstract}



\vskip 1cm

\section{Introduction}\label{Intro} 
Let $H$ be a Hilbert space and let $-A$ be the  
infinitesimal generator of a bounded $C_0$-semigroup $(T_t)_{t\geq 0}$
on $H$.
To any $b\in L^1(\R_+)$, one may associate the operator 
$\Gamma(A,b)\in B(H)$ defined by 
$$
[\Gamma(A,b)](x) = \int_0^\infty b(t) T_t(x)\, dt,\qquad x\in H.
$$
The mapping $b\mapsto \Gamma(A,b)$ is the so-called Hille-Phillips
functional calculus (\cite{hp}, see also \cite[Section 3.3]{haa1})
and we obviously have
$$
\norm{\Gamma(A,b)}\leq C\norm{b}_1,\qquad b\in L^1(\R_+),
$$
where $C=\sup_{t\geq 0}\norm{T_t}$. This holds true as well 
for any  bounded $C_0$-semigroup on Banach space. However we 
focus here on semigroups acting on Hilbert space.

If $(T_t)_{t\geq 0}$ is a contractive semigroup (i.e. $\norm{T_t}\leq 1$
for all $t\geq 0)$ on $H$, then we have the much stronger estimate
$\norm{\Gamma(A,b)}\leq \norm{\widehat{b}}_\infty$
for all $b\in L^1(\R_+)$, where $\widehat{b}$ denotes the Fourier
transform of $b$. This is a semigroup version
of von Neumann's inequality, see \cite[Section 7.1.3]{haa1} 
for a proof. Hence more generally, if
$(T_t)_{t\geq 0}$ is similar to a contractive semigroup, then 
there exists a constant $C\geq 1$ such that
\begin{equation}\label{PB}
\norm{\Gamma(A,b)}\leq C\norm{\widehat{b}}_\infty,\qquad b\in L^1(\R_+).
\end{equation}
However not all negative generators of bounded $C_0$-semigroups satisfy 
such an estimate. Indeed if $A$ is sectorial of type $<\frac{\pi}{2}$, it
follows from \cite[Section 3.3]{haa1} that $A$ satisfies an estimate of the form (\ref{PB}) 
exactly when $A$ has a bounded $H^\infty$-functional calculus,
see Subsection \ref{H-infty} for more on this.

The motivation for this paper is the search
for sharp estimates of $\norm{\Gamma(A,b)}$, 
and of the norms of other functions of $A$, valid for all 
negative generators of bounded $C_0$-semigroups.
A major breakthrough was achieved by Haase \cite[Corollary 5.5]{haa3}
who proved an estimate 
\begin{equation}\label{MH}
\norm{\Gamma(A,b)}\leq C\norm{L_b}_{\B_0},
\end{equation}
where $\norm{\,\cdotp}_{\B_0}$ denotes the norm with respect to a suitable
Besov algebra $\B_0(\C_+)$ of analytic functions, and 
$$
L_b\colon \C_+=\{{\rm Re}(\,\cdotp)>0\}\to \C,\quad
L_b(z)=\int_0^\infty b(t)e^{-tz}\, dt,
$$
is the Laplace transform of $b$. More recently,
Batty-Gomilko-Tomilov \cite{bgt1} (see also \cite{bgt2})
extended Haase's result by providing an
explicit construction of
a bounded functional calculus $\B_0(\C_+)\to B(H)$ associated with $A$,
extending the Hille-Phillips functional calculus.
It is worth mentioning the related
works by Vitse \cite{vit1} and White \cite{white} (see also Subsection \ref{Added}).
We refer to \cite[Appendix]{bgt1} for more information on $\B_0(\C_+)$
and on variants of this Besov algebra.

In this paper we introduce the space $\A(\R)\subset H^\infty(\R)$ defined by
\[
\A(\R)= \Big\{ F= \sum_{k=1}^{\infty}f_k\star h_k\, :\,
(f_k)_{k\in \N}\subset BUC(\R), \, (h_k)_{k\in \N}\subset H^1(\R), \, 
\sum_{k=1}^{\infty}\normeinf{f_k}\norme{h_k}_1 < \infty \Big\},
\]
equipped with the norm $\norme{F}_{\A} = \inf 
\{\sum_{k=1}^{\infty}\normeinf{f_k}\norme{h_k}_1\}$,
where the infimum runs over all sequences  $(f_k)_{k\in \N} \subset BUC(\R)$ 
and $(h_k)_{k\in \N} \subset H^1(\R)$ such that 
$\sum_{k=1}^{\infty}\normeinf{f_k}\norme{h_k}_1<\infty$ and
$F= \sum_{k=1}^{\infty}f_k\star h_k$.
The definition of this space is inspired by Peller's paper \cite{pel1},
where a discrete analogue of $\A(\R)$ was introduced to study 
functions of power bounded operators on Hilbert space. We also
refer to \cite{Wells} for earlier results on this theme.
Furthermore $\A(\R)$ can be regarded as an $H^1(\R)$-version
of the Figa-Talamanca-Herz algebras
$A_p(\R)$, $1<p<\infty$, for which we refer e.g. to 
\cite[Chapter 3]{der}.

We prove in Section \ref{AA0} that $\A(\R)$ is 
indeed a Banach algebra for pointwise
multiplication. Next in Section \ref{CF} we introduce the natural half-plane version  
$\A(\C_+)\subset H^\infty(\C_+)$ of $\A(\R)$, we prove that
it contains $L_b$ for all $b\in L^1(\R_+)$,  and we show
(Corollary \ref{main3}) that whenever $A$ is the negative generator
of a bounded $C_0$-semigroup $(T_t)_{t\geq 0}$ on Hilbert space,
there is a unique bounded homomorphism $\rho_A\colon\A(\C_+)\to B(H)$,
such that
\begin{equation}\label{compatible}
\rho_A(L_b) =\Gamma(A,b),\qquad b\in L^1(\R_+).
\end{equation}
More precisely we show that
$$
\norm{\Gamma(A,b)}\leq\Bigl(\sup_{t\geq 0}\norm{T_t}\Bigr)^2\,\norm{L_b}_{\A},\qquad b\in L^1(\R_+).
$$
This improves Haase's estimate (\ref{MH})
mentioned above. Our work also improves 
\cite[Theorem 4.4]{bgt1} in the Hilbert space case. Indeed we show in Section \ref{Besov} 
that the Besov algebra considered in \cite{haa1,bgt1} is included
in $\A(\C_+)$, with an estimate $\norm{\,\cdotp}_{\A}\lesssim
\norm{\,\cdotp}_{\B_0}$, and we also show that the converse is not true.

In general, our main result (Corollary \ref{main3}) does not hold true on 
non-Hibertian Banach spaces (see the start of Section \ref{Banach}). 
In Section \ref{Banach}, following ideas
from \cite{arn2, arn1, lem1}, we give a Banach space version of Corollary \ref{main3}, using 
the notion of $\gamma$-boundedness. Namely we show that 
if  $A$ is the negative generator
of a bounded $C_0$-semigroup $(T_t)_{t\geq 0}$ on a Banach space
$X$,
then the set $\{T_t\, :\, t\geq 0\}\subset B(X)$ is $\gamma$-bounded 
if and only if  there exists 
a $\gamma$-bounded homomorphism $\rho_A\colon\A(\C_+)\to B(X)$
satisfying (\ref{compatible}). This should be regarded
as a semigroup version of \cite[Theorem 4.4]{lem1}, where a 
characterization of $\gamma$-bounded continuous representations 
of amenable groups was established.

Our results make crucial use of Fourier multipliers on the Hardy
space $H^1(\R)$. Section \ref{MultH1} is devoted to this topic. In particular we 
establish the following result of independent interest: if a bounded operator
$T\colon H^1(\R)\to H^1(\R)$
commutes with translations, then there exists a bounded continuous 
function $m\colon\R_+^*\to\C$ such that $\norm{m}_\infty\leq\norm{T}$
and for any $h\in H^1(\R)$, $\widehat{T(h)} = m\widehat{h}$.

\subsection*{Notation and convention.} 
We use the notations $\R_+=[0,\infty)$, $\R_+^*=(0,\infty)$ 
and $\R_-=(-\infty,0]$ on the real line.

We will use the following open half-planes of 
$\C$, 
$$
\C_+ := \{ z \in \C\, :\, {\rm Re}(z) > 0 \},
\quad
P_+ := \{ z \in \C\, :\,  {\rm Im}(z) > 0 \}, 
\quad
P_- := \{ z \in \C\, :\,  {\rm Im}(z) < 0 \}. 
$$
Also for any real $\alpha\in\R$, we set
\begin{equation}\label{HP0}
{\mathcal H}_\alpha = 
\{ z \in \C\, :\,  {\rm Re}(z) > \alpha \}.
\end{equation}
In particular, ${\mathcal H}_0=\C_+$.

For any $s\in \R$, we let $\tau_s\colon 
L^1(\R)+L^\infty(\R)\to L^1(\R)+L^\infty(\R)$
denote the translation operator defined by
$$
\tau_sf(t) = f(t-s),\qquad t\in\R,
$$
for any $f\in L^1(\R)+L^\infty(\R)$.

The Fourier transform of any $f\in L^1(\R)$ is defined
by
\[
\widehat{f}(u) = \int_{-\infty}^\infty f(t)e^{-itu}\,dt,
\qquad u\in\R.
\]
Sometimes we write $\mathcal{F}(f)$ instead of 
$\widehat{f}$. We will also let $\mathcal{F}(f)$ or
$\widehat{f}$ denote the Fourier transform of any
$f\in L^1(\R)+L^\infty(\R)$. Wherever it makes sense, we will
use $\mathcal{F}^{-1}$ to denote the inverse Fourier transform.

We will use several times the following elementary result (which follows from 
Fubini's theorem and the Fourier inversion theorem).

\begin{lem}\label{Tool} Let $f_1,f_2\in L^1(\R)$ such that 
either $\widehat{f_1}$ or $\widehat{f_2}$ belongs to 
$L^1(\R)$. Then 
$$
\int_{-\infty}^\infty f_1(t)f_2(t)\, dt
\,=\,\frac{1}{2\pi}\,
\int_{-\infty}^\infty \widehat{f_1}(u)\widehat{f_2}(-u)\,
du.
$$
\end{lem}

The norm on $L^p(\R)$ will be denoted by $\norm{\,\cdotp}_p$.
We let $C_0(\R)$ (resp. $BUC(\R)$, resp. 
$C_b(\R)$) denote the Banach algebra  of 
continuous functions on $\R$ which vanish at infinity 
(resp. of bounded and uniformly continuous functions on $\R$, resp.
of bounded continuous functions on $\R$),
equipped with the sup-norm $\norm{\,\cdotp}_\infty$. 
We set
\begin{equation}\label{C00}
C_{00}(\R) = \{f\in L^1(\R)\, :\, \widehat{f}\in L^1(\R)\}.
\end{equation}
This is a dense subspace of $C_0(\R)$.
Further we let $\SR$ denote the Schwartz space on $\R$
and we let 
$M(\R)$ denote the Banach algebra 
of all bounded Borel measures on $\R$.

We will use the identification $M(\R)\simeq C_0(\R)^*$ 
(Riesz's theorem) provided by the duality
pairing
\begin{equation}\label{Riesz}
\langle\mu,f\rangle \,=\, \int_{\tiny{\R}} f(-t)\, d\mu(t),\qquad
\mu\in M(\R),\ f\in C_0(\R).
\end{equation}
The use of a minus sign in this duality pairing will make the 
study of $\A(\R)$ easier.

For any non empty open set $\mathcal{O}\subset\C$, we let
$H^{\infty}(\mathcal{O})$ denote the Banach algebra
of all  bounded analytic functions on $\mathcal{O}$,
equipped with  the sup-norm $\norm{\,\cdotp}_\infty$.

Let $X,Y$ be (complex) Banach spaces. We let 
$B(X,Y)$ denote the Banach space of all bounded operators
$X\to Y$. We simply write $B(X)$ instead of $B(X,X)$,
when $Y=X$. We let $I_X$ denote the identity operator 
on $X$.

The domain of 
an operator $A$ on some Banach space $X$ is 
denoted by ${\rm Dom}(A)$. Its kernel and range are
denoted by ${\rm Ker}(A)$ and ${\rm Ran}(A)$,
respectively. If $z\in\C$  belongs to the resolvent set
of $A$, we let $R(z,A)=(zI_X-A)^{-1}$ denote the corresponding resolvent operator.

\section{Fourier multipliers on $H^1(\R)$}\label{MultH1}
We denote by $H^1(\R)$ the classical Hardy space, 
defined as the closed subspace of $L^1(\R)$ 
of all functions $h$ such that $\widehat{h}(u) = 0$ for all $u\leq 0$. 
For any $1<p< \infty $, we denote by $H^p(\R)$ the closure of $H^1(\R)\cap L^p(\R)$ in $L^p(\R)$.
Also we let $H^\infty(\R)$ denote the $w^*$-closure of $H^1(\R)\cap L^\infty(\R)$ in $L^\infty(\R)$.
We recall (see e.g. \cite{gar}, \cite{hof} or \cite{koo}) that for any $1\leq p\leq\infty$,
$H^p(\R)$
coincides with the subspace of all functions $f\in L^p(\R)$ whose 
Poisson integral ${\mathcal P}[f]\colon P_+\to\C$  is analytic. The following is classical as well.

\begin{lem}\label{H1HInf}
Let $g\in L^\infty(\R)$. Then $g\in H^\infty(\R)$ if and only if $\int_{-\infty}^{\infty} g(t)h(t)\, dt\,=0$
for all $h\in H^1(\R)$.
\end{lem}

\begin{proof}
This statement means that $H^\infty(\R)$ is the annihilator of $H^1(\R)$
in the usual $L^1/L^\infty$-duality. To prove it, it suffices to check the equivalent property that 
$H^1(\R)$ is the pre-annihilator of $H^\infty(\R)$.

If $g\in H^\infty(\R)$  and  $h\in H^1(\R)$, then $gh\in H^1(\R)$ hence 
$\int_{-\infty}^{\infty} g(t)h(t)\, dt\,= \widehat{gh}(0)=0$, which proves one inclusion. 
Now
define $c_u(t)=e^{-iut}$ for all $u\leq 0$ and all $t\in\R$. Then $c_u\in H^\infty(\R)$,
with ${\mathcal P}[c_u](z)=e^{-iuz}$ for all $z\in P_+$. If $h\in L^1(\R)$ 
belongs to the pre-annihilator of $H^\infty(\R)$, then
$\widehat{h}(u)=\int_{-\infty}^{\infty} c_u(t)h(t)\, dt\,=0$ for all $u\leq 0$,
that is, $h\in H^1(\R)$. This proves the reverse inclusion.
\end{proof}

It is well-known that $H^p(\R)$ also coincides with the subspace 
of all functions in $L^p(\R)$ whose (distributional) Fourier transform has support in $\R_+$.
In particular, $H^2(\R)$ is the subspace of all functions in $L^2(\R)$ whose
Fourier transform (regarded as an element of 
$L^2(\R)$)
vanishes almost everywhere on $\R_{-}$. This can be expressed by the identification
\begin{equation}\label{H2=FL2}
\mathcal{F}(H^2(\R)) = L^{2}(\R_+).
\end{equation}
Let $m \in L^{\infty}(\R_+)$. Using (\ref{H2=FL2}), we may associate
\[
\begin{array}{ccccc}
T_m & : & H^2(\R) & \to & H^2(\R) \\
& & h & \mapsto & \mathcal{F}^{-1} (m\widehat{h}), \\
\end{array}
\]
and we have $\norm{T_m}=\norm{m}_\infty$.
The function $m$ is called
the symbol of $T_m$.

Let $1 \leq p < \infty$. Assume that 
$T_m$ is bounded with respect to the $H^p(\R)$-norm, that is, there exists 
a constant $C>0$ such that 
\begin{equation}\label{ineH1mul}
\bignorm{\mathcal{F}^{-1} (m\widehat{h})}_p \leq C\norme{h}_p,\qquad 
h\in H^p(\R)\cap H^2(\R).
\end{equation}
Then since $H^p(\R)\cap H^2(\R)$ is dense in $H^p(\R)$,  
$T_m$ uniquely extends to a bounded operator on $H^p(\R)$
whose norm is the least possible
constant $C$ satisfying \eqref{ineH1mul}. In this case
we keep the same
notation $T_m\colon H^p(\R)\to H^p(\R)$ for this extension.
Operators of this form 
are called bounded Fourier multipliers on $H^p(\R)$.
They form a subspace of $B(H^p(\R))$, that we denote by  
$\MHp$. It is plain that ${\mathcal M}(H^2(\R))\simeq L^{\infty}(\R_+)$ isometrically.
In the sequel we will be mostly interested by $\MH$.

The above definitions parallel the classical definitions of bounded Fourier multipliers
on $L^p(\R)$, that we will use without any further reference.

\begin{ex}\label{taus}
Let $s \in \R$. For all $h \in L^1(\R)$, one has
$\widehat{\tau_s h}(u) = e^{-isu}\widehat{h}(u)$ for all $u\in\R$. Hence
$\tau_s$ maps $H^1(\R)$ into itself. Further $\tau_s$ is 
a bounded Fourier multiplier on $H^1(\R)$, with symbol $m(u) =  e^{-isu}$. 
\end{ex}

In the sequel we say that a bounded operator $T\colon H^1(\R) \to H^1(\R)$ 
commutes with translations if $T\tau_s = \tau_s T$ for each $s\in \R$.
By Example \ref{taus},  any bounded Fourier multiplier on $H^1(\R)$
commutes with translations.
The next result implies that the converse is true and provides a sharp estimate on the symbol of 
an element of $\MH$.

\begin{thm}\label{thmmutlH1}
Let $T \in B(H^1(\R))$ and assume that 
$T$ commutes with translations.
Then there exists a bounded continuous function 
$m \colon \R_+^* \rightarrow \C$ 
such that $\widehat{Th} =m\widehat{h}$ for all 
$h\in H^1(\R)$ (and hence $T = T_m$). In this case, we have
\begin{equation}\label{ineMH1}
\quad \normeinf{m} \leq \norme{T}.
\end{equation}
\end{thm}

\begin{proof}
Except the estimate (\ref{ineMH1}), this statement can be deduced from \cite[pp. 131-132]{BJORK},
and from the fact that  $\{h_1 + h_2(-\cdotp)\, :\, h_1, h_2\in H^1(\R)\}\subset L^1(\R)$ coincides
with the so-called real $H^1$-space
(see also \cite[Theorem 7.31]{GCRF}). We briefly prove the first part of
our statement for completeness and then focus on the proof of  (\ref{ineMH1}).

We use Bochner spaces
and Bochner integrals, for which we refer to \cite{du}.
Let $T \in B(H^1(\R))$ and assume that
$T$ commutes with translations. 

Let $h,g \in H^1(\R)$. The identification 
$L^1(\R^2) = L^1(\R;L^1(\R))$ and the
fact that 
$$
\int_{-\infty}^{\infty}\int_{-\infty}^{\infty}|h(t-s)||g(s)|\,dtds = \norme{h}_1\norme{g}_1\,<\infty
$$
imply that $s\mapsto g(s)\tau_sh$ is an almost everywhere defined function belonging to
the Bochner space $L^1(\R;L^1(\R))$. Since $\tau_sh$ belongs
to $H^1(\R)$ for all $s\in\R$, the latter is actually an element of 
$L^1(\R;H^1(\R))$.
Further its integral (which is an element of $H^1(\R)$)
is equal to
the convolution of $h$ and $g$, that is,
\begin{equation}\label{eq1}
h \star g = \int_{-\infty}^{\infty}\tau_s h\cdotp g(s) \,ds.
\end{equation}
It follows, by the assumption, that 
$$
T(h\star g) = \int_{-\infty}^{\infty} T(\tau_sh)g(s) \,ds\,
=  \int_{-\infty}^{\infty} \tau_s(Th)g(s)\, ds\,=
Th \star g.
$$ 
Likewise $T(h \star g) =h\star Tg$, whence $Th \star g= h \star Tg$. 
Applying the Fourier transform to the latter equality, one obtains
\begin{equation}\label{eq2}
\widehat{Th}\cdotp\widehat{g} = \widehat{h}\cdotp\widehat{Tg}.
\end{equation}

Now let $f \in \SR$ satisfying $f=0$ on $\R_-$ 
and $f > 0$ on $\R_+^*$.
The function $g = \mathcal{F}^{-1}(f)$
belongs to $H^1(\R)$ and we may define $m\colon \R_+^*\to\C$ by
\begin{equation}\label{eq3}
m(u) = \frac{\widehat{Tg}(u)}{\widehat{g}(u)}\,, \qquad u>0.
\end{equation}
Obviously $m$ is continuous. Furthermore it follows from
\eqref{eq2} that for any $h\in H^1(\R)$, we have 
$\widehat{Th}=m\widehat{h}$ on $\R_+^*$. It therefore suffices to show 
that $m$ is bounded and that (\ref{ineMH1}) holds true.

For $h\in H^1(\R)$ and $g\in \S(\R)$, one has
$$
\int_{-\infty}^\infty m(u) \widehat{h}(u) \widehat{g}(u)\, du\,
= \int_{-\infty}^\infty \bigl[\widehat{T(h)}\bigr](u)\widehat{g}(u)\, du\,
=2\pi \int_{-\infty}^\infty [T(h)](-t) g(t)\, dt,
$$
by Lemma \ref{Tool}. This implies that
$$
\Bigl\vert \int_{-\infty}^\infty m(u) \widehat{h}(u)\widehat{g}(u)\, du\,\Bigr\vert\leq
2\pi\norm{T}\norm{h}_1\norm{g}_\infty.
$$
Replacing $g$ by $g_1\star g_2$ for $g_1,g_2\in \S(\R)$ and using 
$\norm{g_1\star g_2}_\infty\leq \norm{g_1}_2\norm{g_2}_2$, we deduce that
$$
\Bigl\vert \int_{-\infty}^\infty m(u) \widehat{h}(u)\widehat{g_1}(u)
\widehat{g_2}(u)\, du\,\Bigr\vert\leq
2\pi\norm{T}\norm{h}_1\norm{g_1}_2\norm{g_2}_2=
\norm{T}\norm{h}_1\norm{\widehat{g_1}}_2\norm{\widehat{g_2}}_2.
$$
This implies $\norm{m\widehat{h}\widehat{g_1}}_2\leq\norm{T}\norm{h}_1
\norm{\widehat{g_1}}_2$ for all 
$g_1\in \S(\R)$ and hence
\begin{equation}\label{Estim}
\norm{m\widehat{h}}_\infty\leq\norm{T}\norm{h}_1, 
\end{equation}
for all $h\in H^1(\R)$.

Now fix a function $\varphi\in\S(\R)$ such that $\varphi\geq 0$,
${\rm Supp}(\widehat{\varphi})\subset[-1,1]$ and $\widehat{\varphi}(0)=1$.
Let $a>0$ and define
$\varphi_a,h_a\colon\R\to\C$ by $\varphi_a(t)=  a \varphi(at)$ and 
$h_a(t) = e^{ita}\varphi_a(t)$. Then 
$$
\widehat{\varphi_a}(u)=\widehat{\varphi}
\Bigl(\frac{u}{a}\Bigr)
\quad\hbox{and}\qquad
\widehat{h_a}(u) =\widehat{\varphi}\Bigl(\frac{u-a}{a}\Bigr),
$$ 
for all $u\in\R$.
Hence $h_a\in H^1(\R)$ and  $\widehat{h_a}(a) =1$.
Furthermore, $\varphi_a\geq 0$,
hence $\norm{\varphi_a}_1=\widehat{\varphi_a}(0) = \widehat{\varphi}(0)=1$.
Since $\norm{h_a}_1=\norm{\varphi_a}_1$, we obtain that 
$\norm{h_a}_1=1$. 
Applying (\ref{Estim}) with $h=h_a$ then yields
$$
\vert m(a)\vert = \vert m(a)\widehat{h_a}(a)\vert \leq 
\norm{m\widehat{h_a}}_\infty\leq\norm{T}.
$$
This proves the boundedness of $m$ and (\ref{ineMH1}).
\end{proof}

\begin{rq1}\label{MHP}
\ 

\smallskip (1) 
Let $1<p<\infty$.
Let $S\colon L^p(\R)\to L^p(\R)$
be a bounded Fourier multiplier. Then $S$ maps 
$H^p(\R)$ into itself and the restriction
$S_{\vert H^p}\colon H^p(\R)\to H^p(\R)$ is a
bounded Fourier multiplier.

Let 
$Q\colon L^p(\R)\to H^p(\R)$ be the Riesz projection and let
$J\colon H^p(\R)\to L^p(\R)$ be the canonical embedding.
Then conversely, for any bounded Fourier multiplier 
$T\colon H^p(\R)\to H^p(\R)$, 
$S=JTQ$ is a bounded Fourier multiplier on $L^p(\R)$,
whose restriction to $H^p(\R)$ coincides with $T$. 
Thus $\MHp$ can be simply regarded
as a subspace of ${\mathcal M}(L^p(\R))$, the space of
bounded Fourier multipliers on $L^p(\R)$.

It is well-known that a bounded operator $L^p(\R)\to L^p(\R)$
is a Fourier multiplier if and only if it commutes with tranlations.
Using the above reasoning, we deduce that 
a bounded operator $H^p(\R)\to H^p(\R)$
belongs to $\MHp$ if and only if it
commutes with translations.

\smallskip (2)
Let $p'=\frac{p}{p-1}$ be the conjugate exponent of 
$1<p<\infty$. 
Using $Q$ again, we see that given any 
$m\in L^\infty(\R_+)$, the operator
$T_m\colon H^2(\R)\to H^2(\R)$ extends 
to a bounded Fourier multiplier on $H^p(\R)$ if and only if it 
extends to a bounded Fourier multiplier on $H^{p'}(\R)$. Thus,
\begin{equation}\label{duality}
\MHp\simeq\mathcal{M}(H^{p'}(\R))
\end{equation}
isomorphically.

\smallskip (3) 
Recall that the bounded Fourier multipliers on $L^1(\R)$
are the operators of the form $h\mapsto \mu\star h$, with 
$\mu\in M(\R)$, and that the norm of the latter operator
is equal to $\norm{\mu}_{M(\tiny{\R})}$ (see e.g.
\cite[Chapter I, Theorem 3.19]{sw}).

For any $\mu\in M(\R)$, let $R_\mu\colon H^1(\R)\to H^1(\R)$
be the restriction of $h\mapsto \mu\star h$ to $H^1(\R)$.
This is a bounded Fourier multiplier whose symbol is equal
to the restriction of $\widehat{\mu}$ to $\R_+^*$.
We set 
\begin{equation}\label{Trivial}
{\mathcal R} = \{ R_\mu\, :\, \mu\in M(\R)\}\,\subset \MH.
\end{equation}
In contrast with the result in part (1) of this remark, we have
$$
{\mathcal R}\not=\MH.
$$ 
Indeed this follows from \cite[Remark (ii)]{Gau}.
\end{rq1}

The following lemma will play a crucial role.

\begin{lem}\label{noMpisnotM1}
For any $1\leq p<\infty$, we have $\MH\subset \MHp$.
\end{lem}

\begin{proof}
By definition we have
$\MH\subset \mathcal{M}(H^{2}(\R))$.
By (\ref{duality}) we may assume that $p\in (1,2)$.
Let $\theta = 2\bigl(1-\frac{1}{p}\bigr)$. Then in the complex interpolation
method, we have
\[
[H^1(\R),H^{2}(\R)]_{\theta} \simeq H^p(\R),
\]
by \cite[Theorem 4.3]{pis4}. The result follows at once.
\end{proof}

\section{Algebras $\A_0$ and $\A$}\label{AA0}
We introduce and study new algebras of functions which will be used in Section \ref{CF} to establish a
functional calculus for negative generators of bounded $C_0$-semigroups on Hilbert space.
The next definitions are insprired by \cite{pel1},
see also the ``Notes and Remarks on Chapter 6" in
\cite{pis1}.

\subsection{Definitions and properties}\label{DandP}

\begin{df}\label{defA}
We let $\mathcal{A}_0(\R)$ (resp. $\mathcal{A(\R)}$) 
be the set of all functions $F
\colon \R \rightarrow \C$ such that there exist two sequences $(f_k)_{k\in \N}$ in $C_0(\R)$  
(resp. $BUC(\R)$) and $(h_k)_{k\in \N}$  in $H^1(\R)$ satisfying
\begin{equation}\label{ineA}
\sum_{k=1}^{\infty} \normeinf{f_k}\norme{h_k}_1 < \infty
\end{equation} 
and
\begin{equation}\label{equA2}
\forall\,s \in \R, \qquad F(s) = \sum_{k=1}^{\infty} (f_k\star h_k)(s).
\end{equation}
\end{df}

For any $f\in L^\infty(\R)$ and any $h\in H^1(\R)$, 
$f\star h$ belongs to $H^\infty(\R)$. To prove this, let $k\in H^1(\R)\cap 
H^\infty(\R)$. By Fubini's theorem, and Lemma \ref{H1HInf},
$$
\int_{-\infty}^\infty (f\star h)(t)k(t)\, dt\,
=\int_{-\infty}^{\infty} f(u)\Bigl(\int_{-\infty}^\infty h(t-u)k(t)\, dt\Bigr)du = 0.
$$
Indeed for all $u\in\R$, $h(\,\cdotp -u)\in H^1(\R)$ hence
the integral of $h(\,\cdotp -u)k$ is equal to $0$.
Since $ H^1(\R)\cap 
H^\infty(\R)$ is dense in $H^1(\R)$, this implies that 
$\int_{-\infty}^\infty (f\star h)(t)k(t)\, dt\,=0$ for all $k\in H^1(\R)$.
By Lemma \ref{H1HInf}, this proves that $f\star h\in H^\infty(\R)$.

We deduce that for any $f\in C_0(\R)$ (resp. $BUC(\R)$) and any $h\in H^1(\R)$,
$f\star h$ belongs to $C_0(\R)\cap H^\infty(\R)$ (resp. 
$BUC(\R)\cap H^\infty(\R)$). Further for any $(f_k)_{k\in \N}$  
and $(h_k)_{k\in \N}$ as in Definition \ref{defA}, we have
$\normeinf{f_k \star h_k } \leq \normeinf{f_k}\norme{h_k}_1$, 
and hence
$\sum_{k=1}^{\infty} \normeinf{f_k\star h_k} < \infty\,$,
by \eqref{ineA}. This 
ensures the convergence of the series in \eqref{equA2}
and implies that 
$$
\A_0(\R) \subset C_0(\R)\cap H^\infty(\R)\qquad\hbox{and}\qquad
\A(\R) \subset BUC(\R)\cap H^\infty(\R).
$$

\begin{df}\label{defNA}
For all $F \in \A_0(\R)$ (resp. $F \in \A(\R)$), we set
\[
\norme{F}_{\A_0} = \inf \Big\{\sum_{k=1}^{\infty} \normeinf{f_k}\norme{h_k}_1 \Big\}
\qquad
\big(\text{resp.} 
\norme{F}_{\A} = \inf \Big\{\sum_{k=1}^{\infty} \normeinf{f_k}\norme{h_k}_1 \Big\} \big),
\]	 
where the infimum runs over all sequences  $(f_k)_{k\in \N}$ 
in $C_0(\R)$ (resp. $BUC(\R)$)  and 
$(h_k)_{k\in \N}$ in $H^1(\R)$ satisfying \eqref{ineA} and \eqref{equA2}. 
\end{df}

It is clear that 
$$
\norme{F}_\infty\leq \norme{F}_{\A},\qquad 
F \in \A(\R).
$$

To show that $\norme{\cdot}_{\A_0}$ and $\norme{\cdot}_{\A}$ are complete norms,
we make a connection with projective tensor products, which will  be useful
throughout this section.

If $X$ and $Y$ are any Banach spaces, the 
projective norm of $\zeta \in X \otimes Y$ is defined by 
\[
\norme{\zeta}_{\wedge} = \inf\Bigl\{ \sum_k \norme{x_k}\norme{y_k} \Bigr\},
\] 
where the infimum runs over all finite families $(x_k)_k$ in $X$ and $(y_k)_k$ in $Y$ satisfying 
$\zeta = \sum_k x_k \otimes y_k\,$.
The completion of $(X\otimes Y, \norme{\,\cdotp}_{\wedge})$, denoted by $X\widehat{\otimes} Y$, 
is called the projective tensor product of $X$ and $Y$. 

Let $Z$ be a third Banach space. 
To any $\ell\in B_2(X\times Y, Z)$, the space of bounded bilinear maps 
from $X\times Y$ into $Z$,
one can associate 
a linear map $\overset{\circ}{\ell}\colon X\otimes Y\to Z$
by the formula
$$
\overset{\circ}{\ell}(x\otimes y)=\ell(x,y),\qquad  x\in X,\, y\in Y.
$$
Then $\overset{\circ}{\ell}$ extends to a  bounded operator (still denoted by)
$\overset{\circ}{\ell}\colon X\widehat{\otimes} Y\to Z$, with $\norm{\overset{\circ}{\ell}}=\norm{\ell}$.
Further the mapping $\ell\mapsto\overset{\circ}{\ell}$ yields 
an isometric identification 
\begin{equation}\label{biliproj}
B_2(X\times Y, Z) \simeq B(X \widehat{\otimes} Y,Z).
\end{equation}
Let us apply the above property in the case 
when $Z=\C$. Using the standard identification $B_2(X\times Y, \C)=B(Y,X^*)$,
we obtain 
an isometric identification 
\begin{equation}\label{isomidenEprojF}
(X\widehat{\otimes}Y)^* \simeq  B(Y,X^*).
\end{equation}
We refer to either \cite[Chapter 8, Theorem 1 $\&$ 
Corollary 2]{du} or \cite[Theorem 2.9]{Ryan}
for these classical facts.

Consider the bilinear map $\sigma \colon C_0(\R) \times H^1(\R) \rightarrow C_0(\R)$ 
defined by 
\begin{equation}\label{sigma}
\sigma(f,h) = f\star h,\qquad f\in C_0(\R),\ h\in H^1(\R).
\end{equation}
Applying (\ref{biliproj}), let
$$
\overset{\circ}{\sigma}\colon C_0(\R)\widehat{\otimes} H^1(\R)\longrightarrow C_0(\R)
$$
be associated with $\sigma$. Then $\A_0(\R) = {\rm Ran}(\overset{\circ}{\sigma})$.
Through the resulting linear isomorphism between $\A_0(\R)$ and
$(C_0(\R)\widehat{\otimes} H^1(\R))/{\rm Ker}(\overset{\circ}{\sigma})$, 
$\norme{\,\cdotp}_{\A_0}$ corresponds to the quotient norm on the latter space (this
follows from either \cite[Chapter 8, Proposition 9 (b)]{du} 
or \cite[Proposition 2.8]{Ryan}). 
Thus $(\A_0(\R),\norme{\,\cdotp}_{\A_0})$ is a Banach space and $\overset{\circ}{\sigma}$
induces  
an isometric identification 
\begin{equation}\label{quotient}
\A_0(\R) \simeq \frac{C_0(\R)
\widehat{\otimes} H^1(\R)}{{\rm Ker}(\overset{\circ}{\sigma})}\,.
\end{equation}
Similarly, $\norme{\,\cdotp}_{\A}$ is a norm on $\A(\R)$
and $(\A(\R),\norme{\,\cdotp}_{\A})$ is a Banach space.

\begin{rq1}\label{rqA0}
It is clear from Definition \ref{defNA} that $\A_0(\R)\subset \A(\R)$
and that for any $F \in \A_0(\R)$, we have
\begin{equation}\label{comparnormA}
\norme{F}_{\A} \leq \norme{F}_{\A_0}.
\end{equation}
We will show in Proposition \ref{normequA0} below that this inequality is actually an equality.
\end{rq1}

\subsection{$\A$ and $\A_0$ are Banach algebras}\label{BA}

\begin{prop}\label{propAlg}
The spaces $\A_0(\R)$ and $\A(\R)$ are Banach algebras for the pointwise multiplication. 
Furthermore, $\A_0(\R)$ is an ideal of $\A(\R)$ and for 
any $F \in \A(\R)$ and $G \in \A_0(\R)$, we have
\begin{equation}\label{ideal}
\norme{FG}_{\A_0} \leq \norme{F}_{\A} \norme{G}_{\A_0}. 
\end{equation}
\end{prop}

\begin{proof}
We will adapt an idea from \cite{pel1}.
Let $f_1, f_2 \in C_0(\R)$ and $h_1,h_2 \in H^1(\R)$. We note that
\begin{equation}\label{eqL1L1}
\int_{-\infty}^{\infty}\int_{-\infty}^{\infty}|h_1(t)||h_2(t+s)|
\,dtds = \norme{h_1}_1\norme{h_2}_1.
\end{equation}
We define, for $s\in \R$, $\varphi_s\colon \R \rightarrow \C$ and $\psi_s 
\colon \R \rightarrow \C$ by
\[
\varphi_s(t) = f_1(t)f_2(t-s) 
\qquad\hbox{and}\qquad
\psi_s(t) = h_1(t)h_2(t+s).
\] 
Since $f_2$ is uniformly continuous, the function
$s \mapsto \varphi_s$ is continuous from $\R$ into $C_0(\R)$.
Thus $s \mapsto \varphi_s$ belongs to 
the Bochner space $L^{\infty}(\R; C_0(\R))$. 
Using (\ref{eqL1L1}) and arguing as 
at the beginning of the proof of Theorem \ref{thmmutlH1}, we see that
$s \mapsto \psi_s$ belongs to $L^1(\R; H^1(\R))$. 
It follows that $s \mapsto \varphi_s\star \psi_s$ is defined almost everywhere
and belongs to $L^{1}(\R; \A_0(\R))$. Moreover, 
\begin{equation}\label{inealgA0}
\int_{-\infty}^{\infty}\norme{\varphi_s\star \psi_s}_{\A_0}ds 
\leq \normeinf{f_1}\normeinf{f_2}\norme{h_1}_1\norme{h_2}_1.
\end{equation}
Indeed, using (\ref{eqL1L1}) and Fubini's theorem,
\begin{align*}
\int_{-\infty}^{\infty}\norme{\varphi_s\star \psi_s}_{\A_0}ds\,
&\leq \int_{-\infty}^{\infty} \norme{\varphi_s}_\infty\norme{\psi_s}_1 ds\\
&\leq \normeinf{f_1}\normeinf{f_2}\int_{-\infty}^{\infty}\norme{\psi_s}_1 ds,
\end{align*}
which is equal to the right-hand side of (\ref{inealgA0}).

The integral of $s \mapsto \varphi_s\star \psi_s$ is an element of
$\A_0(\R)$. We claim that we actually have
\begin{equation}\label{eqalgA0}
\int_{-\infty}^{\infty}\varphi_s\star \psi_s\, ds = (f_1\star h_1)(f_2\star h_2).
\end{equation}
Indeed, using again \eqref{eqL1L1} and Fubini's theorem, we have for all $u \in \R$,
\begin{align*}
\int_{-\infty}^{\infty}(\varphi_s\star \psi_s)(u)ds &= \int_{-\infty}^{\infty}
\int_{-\infty}^{\infty}f_1(t)f_2(t-s)h_1(u-t)h_2(u-t+s)\,dtds \\
& =  \int_{-\infty}^{\infty}f_1(t)h_1(u-t)\int_{-\infty}^{\infty}f_2(t-s)h_2(u-(t-s))\,dsdt \\
&= (f_1\star h_1)(u)(f_2\star h_2)(u).
\end{align*}
Combining (\ref{eqalgA0}) and (\ref{inealgA0}), we obtain
that  $(f_1\star h_1)(f_2\star h_2) \in \A_0(\R)$, with
\[
\norme{(f_1\star h_1)(f_2\star h_2)}_{\A_0} \leq 
\normeinf{f_1}\normeinf{f_2}\norme{h_1}_1\norme{h_2}_1.
\]

Now let $F,G \in \A_0(\R)$ and let $\varepsilon > 0$.
Consider sequences $(f^1_k)_{k\in \N}$, $(f^2_k)_{k\in \N}$ in
$C_0(\R)$ and $(h^1_k)_{k\in \N}$, $(h^2_k)_{k\in \N}$ in $H^1(\R)$ such that 
\[
F = \sum_{k=1}^{\infty} f^1_k\star h^1_k \qquad \text{and} \qquad \sum_{k=1}^{\infty} 
\normeinf{f^1_k}\norme{h^1_k}_1 \leq \norme{F}_{\A_0} + \varepsilon, 
\]
as well as
\[
G = \sum_{k=1}^{\infty} f^2_k\star h^2_k \qquad \text{and} 
\qquad \sum_{k=1}^{\infty} \normeinf{f^2_k}\norme{h^2_k}_1 \leq \norme{G}_{\A_0} + \varepsilon.
\]
Then, using summation in $C_0(\R)$, we have
$$
FG = \sum_{k,l=1}^{\infty}(f^1_k\star h^1_k) (f^2_l\star h^2_l).
$$
Further $(f^1_k\star h^1_k)(f^2_l\star h^2_l)\in\A_0(\R)$ for all $k,l\geq 1$,
and
\begin{align*}
\sum_{k,l=1}^{\infty}\norme{(f^1_k\star h^1_k) (f^2_l\star h^2_l)}_{\A_0}\,
&\leq \sum_{k,l=1}^{\infty} 
\normeinf{f^1_k}\normeinf{f^2_l}\norme{h^1_k}_1\norme{h^2_l}_1\\
& = \Bigl(\sum_{k=1}^{\infty} 
\normeinf{f^1_k}\norme{h^1_k}_1\Bigr)\,\Bigl(\sum_{l=1}^{\infty} 
\normeinf{f^2_l}\norme{h^2_l}_1\Bigr)\\
& \leq \bigl(\norme{F}_{\A_0} + \varepsilon\bigr)
\bigl(\norme{G}_{\A_0} + \varepsilon\bigr).
\end{align*}
Since $\A_0(\R)$ is complete, this shows that $FG\in\A_0(\R)$. Since $\varepsilon > 0$
is arbitrary, we obtain that $\norme{FG}_{\A_0} \leq \norme{F}_{\A_0} \norme{G}_{A_0}$.
This shows that $\A_0(\R)$ is a Banach algebra.

Analogously, $\A(\R)$ is a Banach algebra. 
Moreover if $f_1 \in BUC(\R)$ and $f_2 \in C_0(\R)$, then for each $s\in \R$, 
$\varphi_s\colon t\mapsto f_1(t)f_2(t-s)$ belongs
to $C_0(\R)$. Hence the computations above 
show that $\A_0(\R)$ is an ideal of $\A(\R)$, as well as \eqref{ideal}.
\end{proof}

We note that the Banach algebra $\A_0(\R)$ can be naturally regarded
as an $H^1(\R)$-version of the Figa-Talamanca-Herz algebras
$A_p(\R)$, $1<p<\infty$
(see e.g. \cite[Chapter 3]{der}).

\subsection{Duality results and consequences}\label{Duality}
The main aim of this subsection is to identify $\A_0(\R)^*$ with a subspace of 
$\MH$, the space of bounded Fourier multipliers on $H^1(\R)$. 
This requires the use of duality tools.

We let $H^1_{-}(\R)$ (resp. $H^\infty_{-}(\R)$) be the subspace
of all $f$ in $L^1(\R)$ (resp. $L^\infty(\R)$) such that 
$f(-\,\cdotp)$ belongs to $H^1(\R)$ (resp. $H^\infty(\R)$).
Recall the identification $M(\R)\simeq C_0(\R)^*$ provided by 
(\ref{Riesz}) and 
regard $H^1(\R)\subset L^1(\R)\subset M(\R)$ in the usual way. 
We have
\begin{equation}\label{Ortho-H1}
(C_0(\R)\cap H^\infty_{-}(\R))^\perp = H^1(\R).
\end{equation}
Indeed the inclusion $\subset$ follows from \cite[Chapter II, Theorem 3.8]{gar}
whereas the reverse inclusion follows from Lemma \ref{H1HInf}.

Set $Z_0 : = C_0(\R)\cap H^\infty_{-}(\R)$ for convenience, then the above result 
yields an isometric identification
\begin{equation}\label{H1Dual}
\biggr(\frac{C_0(\R)}{Z_0}\biggl)^*\,\simeq\, H^1(\R).
\end{equation}
In the sequel, we let $\dot{f}\in \frac{C_0(\R)}{Z_0}$ denote the
class of any $f\in C_0(\R)$.

We note that for any $f\in C_0(\R)$, $h\in H^1(\R)$ and $s\in\R$, we have
\begin{equation}\label{Form1}
(f\star h)(s) = \langle \tau_{-s}h, f\rangle.
\end{equation}
Thus $f\star h=0$ for any $f\in Z_0$ and $h\in H^1(\R)$.
The bilinear map $\sigma\colon C_0(\R)\times H^1(\R)\to C_0(\R)$
defined by (\ref{sigma}) therefore induces a bilinear 
map $\delta \colon \frac{C_0(\R)}{Z_0}\,\times H^1(\R)\to C_0(\R)$ given by
$$
\delta(\dot{f},h) = f\star h, \qquad f\in C_0(\R),\ h\in H^1(\R).
$$
Let 
\begin{equation}\label{delta}
\overset{\circ}{\delta} \colon \frac{C_0(\R)}{Z_0} \widehat{\otimes} H^1(\R) \longrightarrow C_0(\R)
\end{equation}
be the bounded map induced by $\delta$. Then the argument leading to (\ref{quotient}) shows as well
that $\A_0(\R)$ is equal to the range of $\overset{\circ}{\delta}$ and that
\begin{equation}\label{quotient2}
\A_0(\R) \simeq \Bigl(\frac{C_0(\R)}{Z_0}\widehat{\otimes} H^1(\R)\Bigr)\Big/{\rm Ker}(\overset{\circ}{\delta}).
\end{equation}

Since $H^1(\R)$ is a dual space, by (\ref{H1Dual}), $B(H^1(\R))$ is a dual space 
as well. Indeed applying (\ref{isomidenEprojF}), we have an isometric identification
\begin{equation}\label{BH1-dual}
B(H^1(\R))\simeq \Bigl(\frac{C_0(\R)}{Z_0}\widehat{\otimes} H^1(\R)\Bigr)^*.
\end{equation}
If we unravel the identifications leading to (\ref{BH1-dual}), we obtain that
the latter is given by
\begin{equation}\label{BH1-dual+}
\langle T, \dot{f}\otimes h\rangle\,=\, \int_{-\infty}^\infty 
(Th)(t)f(-t)\, dt,
\qquad T\in B(H^1(\R)),\ f\in C_0(\R),\ h\in H^1(\R).
\end{equation}

\begin{lem}\label{MH-closed}
The space $\MH$ is $w^*$-closed in $B(H^1(\R))$.
\end{lem}

\begin{proof}
According to Theorem \ref{thmmutlH1}, an operator $T\in B(H^1(\R))$ belongs
to $\MH$ if and only if $\tau_sT=T\tau_s$ for all $s\in\R$. Hence it suffices to 
show that the maps $T\mapsto\tau_sT$ and $T\mapsto T\tau_s$ are $w^*$-continuous
on $B(H^1(\R))$, for all $s\in\R$.

Fix some $s\in\R$ and note that the mapping $\frac{C_0(\R)}{Z_0}\to \frac{C_0(\R)}{Z_0}$
taking $\dot{f}$ to  $\dot{\overbrace{\tau_s f}}$ for any $f$ in $C_0(\R)$ is a well-defined isometric isomorphism.
This implies the existence of an isometric isomorphism 
$$
w_s\colon \frac{C_0(\R)}{Z_0}\widehat{\otimes} H^1(\R)\longrightarrow 
\frac{C_0(\R)}{Z_0}\widehat{\otimes} H^1(\R)
$$
such that $w_s(\dot{f}\otimes h)= \dot{\overbrace{\tau_s f}}\otimes h$ for 
all $f\in C_0(\R)$ and all $h\in H^1(\R)$.
It readily follows from
(\ref{BH1-dual+}) that for any $T\in B(H^1(\R))$, and any $f,h$ as above, we have
$$
\langle \tau_s T, \dot{f}\otimes h\rangle = \langle T, \dot{\overbrace{\tau_s f}}\otimes h\rangle.
$$
Thus $w_s^*\colon   B(H^1(\R))\to  B(H^1(\R))$ coincides with 
$T\mapsto\tau_sT$. The latter is therefore $w^*$-continuous.
The proof 
that $T\mapsto T\tau_s$ is $w^*$-continuous is similar.
\end{proof}

We introduce 
$$
\PM = \overline{\rm Span}^{w^*}\bigl\{\tau_s\, :\, s\in\R\bigr\}\,\subset B(H^1(\R)).
$$
A direct consequence of Lemma \ref{MH-closed} is that
\begin{equation}\label{MH-inclusion}
\PM\subset \MH.
\end{equation}

\begin{lem}\label{PM}
Recall the mapping $\overset{\circ}{\delta}$ from (\ref{delta}). Then 
$\bigl[{\rm Ker}(\overset{\circ}{\delta})\bigr]^\perp =\PM$.
\end{lem}

\begin{proof} Let $s\in \R$ and
let $\Phi\in \frac{C_0(\R)}{Z_0}\widehat{\otimes} H^1(\R)$. According
to \cite[Proposition 2.8]{Ryan}, we may choose
two sequences $(f_k)_{k\in \N}$ in
$C_0(\R)$ and $(h_k)_{k\in \N}$ in $H^1(\R)$ such that
$$
\sum_{k=1}^\infty \norme{f_k}_\infty\norme{h_k}_1<\infty
\qquad\hbox{and}\qquad \Phi=\sum_{k=1}^\infty \dot{f_k}\otimes h_k.
$$
Then by (\ref{Form1}),
$$
\bigl[\overset{\circ}{\delta}(\Phi)\bigr](s) = \sum_{k=1}^\infty (f_k\star h_k)(s) 
= \sum_{k=1}^\infty \langle \tau_{-s} h_k, f_k\rangle
= \langle\tau_{-s},\Phi\rangle.
$$
This shows that
$$
{\rm Span}\bigl\{\tau_s\, :\, s\in\R\bigr\}_\perp = 
{\rm Ker}(\overset{\circ}{\delta}). 
$$
The result follows at once.
\end{proof}

By standard duality and (\ref{quotient2}), the dual space $\A_0(\R)^*$ may be identified
with $\bigl[{\rm Ker}(\overset{\circ}{\delta})\bigr]^\perp$. Applying Lemma \ref{PM} and
(\ref{BH1-dual+}), we
therefore obtain the following.

\begin{thm}\label{DualA0}
\ 

\begin{itemize} 
\item [(1)] For any $T\in\PM$, there exists a unique $\eta_T\in\A_0(\R)^*$ such that
$$
\langle\eta_T, f\star h\rangle = \int_{-\infty}^\infty (Th)(t) f(-t)\, dt
$$
for any $f\in C_0(\R)$ and any $h\in H^1(\R)$.
\item [(2)] The mapping $T\mapsto \eta_T$ induces a $w^*$-homeomorphic and
isometric identification
$$
\A_0(\R)^* \simeq\PM.
$$
\end{itemize}
\end{thm}

\begin{rq1}  
Recall ${\mathcal R}$ from (\ref{Trivial}). It turns out that
$$
\PM =\overline{{\mathcal R}}^{w^*}.
$$
Indeed for any $s\in \R$, $\tau_s\colon H^1(\R)\to H^1(\R)$ is equal 
to the convolution by the Dirac mass at $s$, which yields $\subset$.
To show the converse inclusion, we observe that for 
any $\mu\in M(\R)$ and any
$\Phi\in \frac{C_0(\R)}{Z_0}\widehat{\otimes} H^1(\R)$, we have
\begin{equation}\label{Connection}
\langle \mu, \overset{\circ}{\delta}(\Phi)\rangle 
=\langle R_\mu,\Phi\rangle.
\end{equation}
Here we use 
the identification $M(\R)\simeq C_0(\R)^*$ given 
by (\ref{Riesz}) on the left-hand side and
we use (\ref{BH1-dual}) on the right-hand side.
To prove this identity, let $f\in 
C_0(\R)$ and $h\in H^1(\R)$. Then
\begin{align*}
\langle \mu, \overset{\circ}{\delta}(\dot{f}\otimes h)\rangle \,
& = \langle \mu, f\star h\rangle \\
& = \int_{\tiny\R}\int_{-\infty}^\infty f(-t)h(t-s)\, dt\,
d\mu(s)\\
& = \langle \mu\star h, f \rangle.
\end{align*}
This shows (\ref{Connection}) when $\Phi=\dot{f}\otimes h$.
By linearity, this implies (\ref{Connection}) for 
$\Phi\in \frac{C_0(\R)}{Z_0} \otimes H^1(\R)$. Then by density,
we deduce (\ref{Connection}) for all $\Phi$.

It clearly follows from (\ref{Connection}) that 
${\mathcal R}\subset\bigl[{\rm Ker}(\overset{\circ}{\delta})\bigr]^\perp$.
By Lemma \ref{PM}, this yields $\supset$.
\end{rq1}

We now give a few consequences of the above duality results.

\begin{prop}\label{lemcont}
For any 
$b$ in  $L^1(\R_+)$, the function $\widehat{b}(-\,\cdotp)\colon\R\to\C\,$
belongs to $\A_0(\R)$ and
$$
\bignorm{\widehat{b}(-\,\cdotp)}_{\A_0}\leq\norme{b}_1.
$$
Moreover the mapping $L^1(\R_+)\to \A_0(\R)$ taking 
$b$ to $\widehat{b}(-\,\cdotp)$ is a Banach algebra 
homomorphism with dense range.
\end{prop}

\begin{proof}
We will use the space $C_{00}(\R)$ defined by
(\ref{C00}). Let $C_{00}\star H^1\subset \A_0(\R)$
be the linear span of the functions $f\star h$, for $f\in C_{00}(\R)$ and
$h\in H^1(\R)$. By definition of $C_{00}(\R)$,
the Fourier transform maps $C_{00}\star H^1$ into $L^1(\R_+)$. Then we 
consider
$$
{\mathcal C}_{0,1}=\{\widehat{F}\, :\, F\in C_{00}
\star H^1\}\subset L^1(\R_+).
$$
Let $b\in C^\infty(\R_+^*)$ with compact support. Let  
$c\in C^\infty(\R_+^*)$ with compact 
support
such that $c\equiv 1$ on the support of $b$,
so that $b=bc$.
Then ${\mathcal F}^{-1}(b)\in C_{00}(\R)$, ${\mathcal F}^{-1}(c)\in H^1(\R)$ and the Fourier transform 
of ${\mathcal F}^{-1}(b)\star {\mathcal F}^{-1}(c)$ is equal to $b$. Thus $b\in {\mathcal C}_{0,1}$. 
Consequently, ${\mathcal C}_{0,1}$ is dense in $L^1(\R_+)$.

Let $b\in {\mathcal C}_{0,1}$ and let $F=\widehat{b}(-\,\cdotp)$, so that
$$
\widehat{F}\,=(2\pi)b.
$$
Take finite families 
$(f_k)_k$ in $C_{00}(\R)$ and $(h_k)_k$ in $H^1(\R)$ such that $F=\sum_k 
f_k\star h_k$. Pick $\eta\in\A_0(\R)^*$ such that 
$\norme{\eta}=1$ and $\norme{F}_{\A_0}=
\langle\eta,F\rangle$. By Theorem \ref{DualA0}, there exists
$T\in\MH$ such that $\norme{T}_{B(H^1)}=1$ and for any $k$,
$$
\langle\eta, f_k\star h_k\rangle = \int_{-\infty}^\infty
(Th_k)(u) f_k(-u)\, du.
$$
By Theorem \ref{thmmutlH1}, the symbol $m$ of the multiplier
$T$ satisfies $\norme{m}_\infty\leq 1$.
Moreover by Lemma \ref{Tool}, we have
\begin{align*}
\int_{-\infty}^\infty (Th_k)(u) f_k(-u)\, du\,
&= \,\frac{1}{2\pi}\int_{0}^\infty \widehat{Th_k}(t)
\widehat{f_k}(t)\, dt\\
&= \,\frac{1}{2\pi}\int_{0}^\infty m(t)
\widehat{h_k}(t)\widehat{f_k}(t)\, dt,
\end{align*}
for all $k$. Summing over $k$, we deduce
that
$$
\langle\eta,F\rangle = \,\frac{1}{2\pi}\,
\sum_k  \int_{0}^\infty m(t)
\widehat{h_k}(t)\widehat{f_k}(t)\, dt\,
= \,\frac{1}{2\pi}\,\int_{0}^\infty m(t)
\widehat{F}(t)\, dt,
$$
and hence
$$
\langle\eta,F\rangle = \int_{0}^\infty m(t)b(t)\,dt.
$$
We deduce that
$$
\norm{F}_{\A_0}\leq\norm{b}_1.
$$
Since ${\mathcal C}_{0,1}$ is dense in $L^1(\R_+)$
and $\A_0(\R)$ is complete, this estimate
implies that $\widehat{b}(-\,\cdotp)\colon\R\to\C\,$
belongs to $\A_0(\R)$ for any $b\in L^1(\R_+)$, with
$\bignorm{\widehat{b}(-\,\cdotp)}_{\A_0}\leq\norme{b}_1$.

It is plain that  $b\mapsto \widehat{b}(-\,\cdotp)$ is a Banach algebra 
homomorphism. Its range 
contains $C_{00}\star H^1$ and the latter is dense in 
$\A_0(\R)$, because $C_{00}(\R)$ is dense in $C_0(\R)$. 
\end{proof}

\begin{rq1}\label{DualA0+}
Let $\eta\in\A_0(\R)^*$, let 
$T\in \MH$ be the multiplier associated with $\eta$, according to Theorem 
\ref{DualA0}, and let $m\in C_b(\R_+^*)$ be the symbol of $T$. Then it
follows from the previous result and its proof that for all $b\in L^1(\R_+)$,
$$
\langle\eta,\widehat{b}(-\,\cdotp)\rangle = 
\int_{0}^\infty m(t)b(t)\,dt.
$$
\end{rq1}

\begin{rq1}\label{Rational}
For any $\lambda\in P_-$, we let 
$$
b_\lambda(t)= ie^{-i\lambda t},\qquad t>0.
$$
Then $b_\lambda\in L^1(\R_+)$ and we have
$$
\widehat{b_\lambda}(-u) = \,\frac{1}{\lambda - u},\qquad
u\in\R.
$$
Applying Proposition \ref{lemcont}, we obtain 
that $(\lambda - \,\cdotp)^{-1}$ belongs to $\A_0(\R)$
for any $\lambda\in P_-$. Since $\A_0(\R)$ is a Banach
algebra, this implies that any rational function
$F\colon \R\to\C\,$ with degree $deg(F)\leq -1$ and poles in $P_-$ 
belongs to $\A_0(\R)$.
\end{rq1}

We can now strengthen Remark \ref{rqA0} as follows.

\begin{prop}\label{normequA0}
For any $F \in \A_0(\R)$, we have
\[
\norme{F}_{\A_0}=\norme{F}_{\A}.
\]	
\end{prop}

\begin{proof}
For any $N\in\N$, let $G_N\colon\R\to\C$ be defined by
$$
G_N(u)=\,\frac{N}{N-iu}\,,\qquad u\in\R.
$$
Then $G_N = \widehat{Ne^{-N\,\cdotp}}(-\,\cdotp)$. We note that
the sequence $(Ne^{-N\,\cdotp})_{N\in\N}$ is
a contractive approximate unit of $L^1(\R_+)$. It therefore follows from
Proposition \ref{lemcont} that 
$(G_N)_{N\in\N}$ is a contractive approximate unit of $\A_0(\R)$.

Let $F\in\A_0(\R)$. By Proposition \ref{propAlg}, we have
\begin{equation}\label{inenormA0A}
\norme{FG_N}_{\A_0} \leq \norme{F}_{\A}\norme{G_N}_{\A_0} \leq 
\norme{F}_{\A}.
\end{equation}
Moreover $FG_N\to F$ in $\A_0(\R)$ hence $\norme{FG_N}_{\A_0}\to
\norme{F}_{\A_0}$ when $N\to\infty$. We deduce that
\[
\norme{F}_{\A_0} \leq  \norme{F}_{\A}.
\]
Combining with \eqref{comparnormA}, we obtain
$\norme{F}_{\A_0} =  \norme{F}_{\A}$. 
\end{proof}

\begin{rq1}\label{densities}
According to 
\cite[Chapter II, Theorem 3.8]{gar}, $H^1_{-}(\R)\subset M(\R)$ is the annihilator 
of $\bigl\{(\lambda-\,\cdotp)^{-1}\, :\, \lambda\in P_{-}\bigr\}\subset
C_0(\R)$.
Hence we deduce from
Remark \ref{Rational} that $H^1_{-}(\R)$ contains $\A_0(\R)^\perp$. 
By (\ref{Ortho-H1}), we have $(C_0(\R)\cap H^\infty(\R))^\perp=H^1_{-}(\R)$.
Hence $\A_0(\R)^\perp=(C_0(\R)\cap H^\infty(\R))^\perp$.
This implies that $\A_0(\R)$
is dense in $C_0(\R)\cap H^\infty(\R)$.

Using Proposition \ref{lemcont}, or repeating the above argument, we also obtain 
that the space $\{\widehat{b}(-\,\cdotp)\, :\, b\in L^1(\R_+)\}$
is dense in $C_0(\R)\cap H^\infty(\R)$.
\end{rq1}

\subsection{Half-plane versions}\label{HP-Versions} 
For the purpose of studying
functional calculus in the next three sections,
we now introduce half-plane versions $\A_0(\C_+)$ and
$\A(\C_+)$ of $\A_0(\R)$ and $\A(\R)$, respectively.

Let $F\in H^\infty(\R)$. We
may consider its Poisson integral ${\mathcal P}[F]\colon P_+\to\C$ and the latter is a bounded
holomorphic function (see e.g. \cite[Sect. I.3]{gar}). Then we define 
$$
\widetilde{F}\colon \C_+\longrightarrow\C
$$
by setting
$$
\widetilde{F}(z) = {\mathcal P}[F](iz),\qquad z\in\C_+.
$$
Note that the mapping $F\mapsto \widetilde{F}$ is an isometric algebra
isomorphism of $H^\infty(\R)$ onto $H^\infty(\C_+)$.

We set
$$
\A_0(\C_+) = \bigl\{\widetilde{F}\, :\, F\in\A_0(\R)\bigr\}
\qquad\hbox{and}\qquad 
\A(\C_+) = \bigl\{\widetilde{F}\, :\, F\in\A(\R)\bigr\}.
$$
We equip these spaces with the norms induced by $\A_0(\R)$ and $\A(\R)$, respectively.
That is, we set $\norm{\widetilde{F}}_{\A_0}= \norm{F}_{\A_0}$ (resp.
$\norm{\widetilde{F}}_{\A}= \norm{F}_{\A}$) for any $F\in\A_0(\R)$ (resp.
$F\in\A(\R)$). By construction, we have 
\begin{equation}\label{Inclusions}
\A_0(\C_+)\subset \A(\C_+)
\qquad\hbox{and}\qquad 
\A(\C_+)\subset H^\infty(\C_+).
\end{equation}
It clearly follows from Proposition \ref{propAlg} and from the 
multiplicativity of the mapping $F\mapsto \widetilde{F}$
that
$\A(\C_+)$ is a Banach algebra for pointwise multiplication
and that $\A_0(\C_+)$ is an ideal of $\A(\C_+)$.
Further
the second inclusion in (\ref{Inclusions}) is 
contractive  and
it follows 
from Proposition \ref{normequA0}
that the first inclusion in (\ref{Inclusions}) is 
an isometry.

Let $b\in L^1(\R_+)$ and consider  $F=\widehat{b}(-\,\cdotp)\colon\R\to\C$.
Then $\widetilde{F}$ coincides with the Laplace transform $L_b\colon\C_+\to\C$
of $b$ defined by
\begin{equation}\label{Lb}
L_b(z)=\,\int_{0}^{\infty} e^{-tz} b(t)\, dt,\qquad z\in\C_+.
\end{equation}
As a consequence of Proposition \ref{lemcont}, we therefore obtain the following.

\begin{lem}\label{Laplace}
For any $b\in L^1(\R_+)$,  $L_b$ belongs to 
$\A_0(\C_+)$ and $\norm{L_b}_{\A_0}\leq \norm{b}_1$. 
Moreover the mapping $b\mapsto L_b$ is a Banach algebra 
homomorphism from $L^1(\R_+)$ into $\A_0(\C_+)$, and
the space $\{L_b\,:\, b\in L^1(\R_+)\}$ is dense in $\A_0(\C_+)$.
\end{lem}

\section{Functional calculus on $\A_0$ and $\A$}\label{CF}
The goal of this section is to construct a bounded functional calculus for the negative generator
of a bounded $C_0$-semigroup on Hilbert space, defined on 
$\A(\C_+)$.

\subsection{Half-plane holomorphic functional calculus}\label{HPFC}
We need some background
on the half-plane holomorphic functional calculus introduced by 
Batty-Haase-Mubeen in \cite{bat-haa}, to which we refer for details.

Let $X$ be an arbitrary Banach space.
Let $\omega\in\R$ and recall the definition of ${\mathcal H}_{\omega}$ from (\ref{HP0}).
Let $A$ be a closed and densely defined operator on
$X$ 
such that the spectrum of $A$ is included 
in the closed half-plane $\overline{{\mathcal H}_{\omega}}$ and
\begin{equation}\label{HP}
\forall\,\alpha<\omega,\qquad
\sup \bigl\{\norme{R(z,A)}\, :\, {\rm Re}(z) \leq \alpha \bigr\} < \infty.
\end{equation}

Consider the auxiliary algebra
\[
\mathcal{E}({\mathcal H}_{\alpha}) := \bigl\{\varphi 
\in H^{\infty}({\mathcal H}_{\alpha})
\,:\,\exists s>0,\, \vert \varphi(z)\vert  = O(|z|^{-(1+s)}) \text{ as } |z| \rightarrow \infty\bigr\},
\]
for any $\alpha<\omega$. 
For any $\varphi\in \mathcal{E}({\mathcal H}_{\alpha})$ and for any $\beta\in(\alpha,\omega)$, 
the assumption (\ref{HP}) insures that the integral
\begin{equation}\label{fofA}
\varphi(A)  := \frac{-1}{2\pi}\,\int_{-\infty}^{\infty}\varphi(\beta+is)R(\beta+is,A)\,ds
\end{equation}
is absolutely convergent in $B(X)$.
Further its value is independent of the choice of $\beta$ (this is due
to Cauchy's Theorem for vector-valued holomorphic functions) and
the mapping $\varphi\mapsto \varphi(A)$ in an algebra homomorphism
from $\mathcal{E}({\mathcal H}_{\alpha})$ into $B(X)$. This definition is compatible with the
usual rational functional calculus; indeed for any $\mu\in\C$ with
${\rm Re}(\mu)<\alpha$ and any integer $m\geq 2$, the function 
$$
e_{\mu,m}\colon z\mapsto (\mu -z)^{-m},
$$
which belongs to $\mathcal{E}({\mathcal H}_{\alpha})$, 
satisfies $e_{\mu,m}(A)=R(\mu,A)^m$.

Let $\varphi\in H^{\infty}({\mathcal H}_{\alpha})$. 
We can define a closed, densely defined, operator 
$\varphi(A)$ by regularisation, as follows (see \cite{bat-haa} and \cite{haa1} for more on such constructions).
Take $\mu\in\C$ with
${\rm Re}(\mu)<\alpha$, and set $e=e_{\mu,2}$. Then 
$e\varphi \in \mathcal{E}({\mathcal H}_{\alpha})$ and 
$e(A) = R(\mu,A)^2$ is injective. Then $\varphi(A)$ is defined by 
\begin{equation}\label{eqregul}
\varphi(A) = e(A)^{-1}(e\varphi)(A),
\end{equation}
with domain ${\rm Dom}(\varphi(A))$
equal to the space of all $x\in X$ such that $[(e\varphi)(A)](x)$ belongs
to the range of $e(A)$ $(= {\rm Dom}(A^2))$. 
It turns out that this definition does not depend on the choice of $\mu$.

The half-plane holomorphic functional calculus $\varphi\mapsto \varphi(A)$
satisfies the following ``Convergence Lemma'', provided
by \cite[Theorem 3.1]{bat-haa}.

\begin{lem}\label{CVLemma}
Assume that $A$ satisfies (\ref{HP}) for some $\omega\in\R$ and
let $\alpha<\omega$. Let $(\varphi_i)_i$ be a net of 
$H^\infty({\mathcal H}_\alpha)$
such that $\varphi_i(A)\in B(X)$ for all $i$
and let $\varphi\in H^\infty({\mathcal H}_\alpha)$ such that $\varphi_i\to \varphi$
pointwise on $R_\alpha$, when $i\to\infty$. If
$$
\sup_i\norm{\varphi_i}_{H^\infty(R_\alpha)}\,<\infty
\qquad\hbox{and}\qquad 
\sup_i\norm{\varphi_i(A)}_{B(X)}\,<\infty,
$$
then $\varphi(A)\in B(X)$ and for any $x\in X$, 
$[\varphi_i(A)](x)\to [\varphi(A)](x)$
when $i\to\infty$.
\end{lem}

Let $(T_t)_{t\geq 0}$ be a bounded $C_0$-semigroup on $X$ and
let $-A$ denote its infinitesimal generator. For any
$b\in L^1(\R_+)$, we set
$$
\Gamma(A,b) :=\,\int_0^\infty b(t)\, T_t\, dt,
$$
where the latter integral is defined in the strong sense.
The mapping $b\mapsto \Gamma(A,b)$ is the so-called Hille-Phillips
functional calculus. We refer to \cite[Section 3.3]{haa1} for information and background. We recall that
this functional calculus is a Banach algebra homomorphism
$$
L^1(\R_+)\longrightarrow B(X).
$$

We now provide a compatibility result between the half-plane holomorphic functional calculus
and the Hille-Phillips functional calculus. This kind of compatibility properties
is very common in functional calculus (see e.g. \cite[Section 3.3]{haa1}) and we follow a classical approach.
Note that any $A$ as above satisfies (\ref{HP}) 
for $\omega=0$. Thus for any $\varepsilon >0$, the operator $A+\varepsilon$
satisfies (\ref{HP}) 
for $\omega=\varepsilon$. For any
$b\in L^1(\R_+)$, this allows us
to define $L_b(A+\varepsilon)$, where $L_b$
is the Laplace transform defined by (\ref{Lb}).

\begin{lem}\label{lemcompatibility}
Let $b\in L^1(\R_+)$ and let $-A$ be the generator
of a bounded $C_0$-semigroup on $X$. Then for any $\varepsilon>0$, we have
$$
L_b(A+\varepsilon) = \Gamma(A+\varepsilon,b).
$$
\end{lem}

\begin{proof}
Let $\varepsilon >0$, let $\beta\in(0,\varepsilon)$, 
and let $b\in L^1(\R_+)$.

First suppose that $L_b \in \mathcal{E}({\mathcal H}_0)$. 
Then by the Laplace formula, 
$$
L_b(A+\varepsilon) 
=\, 
\frac{1}{2\pi}
\int_{-\infty}^{\infty}
L_b(\beta + is)
\Big(\int_{0}^{\infty}e^{ist}e^{\beta t}e^{-\varepsilon t}T_t\,dt\Big)ds.
$$
Since $L_b \in \mathcal{E}({\mathcal H}_0)$, the function $s\mapsto L_b(\beta +is)$ belongs to 
$L^1(\R)$. Hence by Fubini's theorem,
$$
L_b(A+\varepsilon) = 
\,\frac{1}{2\pi}\int_{0}^{\infty}e^{\beta t}
\Big(\int_{-\infty}^{\infty}L_b(\beta + is)e^{ist}\,ds\Big)
\,e^{-\varepsilon t}T_t\,dt.
$$
By the Fourier inversion Theorem, we have
\[
\frac{1}{2\pi}\int_{-\infty}^{\infty}L_b(\delta +is)e^{ist}\,ds
 = e^{-\beta t}b(t)
\]
for each $t>0$. We deduce that
\[
L_b(A+\varepsilon) = \int_{0}^{\infty}b(t)e^{-\varepsilon t}T_t\, dt = \Gamma(A+\varepsilon,b).
\]

For the general case, let us consider $e \in \mathcal{E}({\mathcal H}_0)$
defined by $e(z) = (1+z)^{-2}$. We note that $e$
is the Laplace transform of the function $c\in L^1(\R_+)$ defined by $c(t) = te^{-t}$.
The product $eL_b$, which belongs to 
$\mathcal{E}({\mathcal H}_0)$, is therefore the Laplace transform of $b\star c$. Hence by the first part of this proof,
\[ 
(eL_b)(A+\varepsilon) = \Gamma(A+\varepsilon,b\star c).
\]
The multiplicativity of the Hille-Phillips functional calculus yields 
$\Gamma(A+\varepsilon,b\star c)=\Gamma(A+\varepsilon,c)\Gamma(A+\varepsilon, b)$. Further
$e(A+\varepsilon) = \Gamma(A+\varepsilon, c)$ by the Laplace formula. Thus we
have
$$
e(A+\varepsilon)\Gamma(A+\varepsilon,b) = (eL_b)(A+\varepsilon).
$$
Applying \eqref{eqregul}, we obtain that 
$L_b(A+\varepsilon) = \Gamma(A+\varepsilon,b)$ as required. 
\end{proof}

\subsection{Functional calculus on $\A_0(\C_+)$}\label{MAIN}
Throughout this subsection, we fix a Hilbert space
$H$, we let $(T_t)_{t \geq 0}$ be a bounded $C_0$-semigroup on $H$
and we let $-A$ denote its infinitesimal generator. We set
\[
C := \sup\{ \norme{T_t}\, :\, t\geq 0 \}.
\]

For any $f\in C_{00}(\R)$ and any $h \in H^1(\R)$, the function 
$$
b = (2\pi)^{-1}\widehat{f}\widehat{h}
$$
belongs to $L^1(\R_+)$ and  
we have $\widehat{b}(-\,\cdotp)= f\star h$. Consequently, 
\[
(f\star h)^{\sim} = L_b.
\]
Further we have the following key estimate, which is
inspired by \cite[Proposition 4.16]{pis1}.

\begin{lem}\label{key}
For any $f\in C_{00}(\R)$ and any $h \in H^1(\R)$, 
\begin{equation}\label{inelemcalfon}
\bignorm{\Gamma\bigl(A, (2\pi)^{-1}\widehat{f}\widehat{h}\bigr)} \leq  
C^2\normeinf{f}\norme{h}_1.
\end{equation}
\end{lem}

\begin{proof}
We fix $f\in C_{00}(\R)$. 
Let $w,v \in H^2(\R)\cap \mathcal{S}(\R)$ and let $h=wv$. By definition,
\[
\Gamma\bigl(A, (2\pi)^{-1}\widehat{f}\widehat{h}\bigr)= \,\frac{1}{2\pi}\,
\int_{0}^{\infty}\widehat{f}(t)\widehat{wv}(t)T_t\, dt. 
\]
By assumption, $\widehat{w}$ and $\widehat{v}$ belong to $L^1(\R_+)$
hence $\widehat{wv} = (2\pi)^{-1} \widehat{w}\star\widehat{v}$ belongs to $L^1(\R_+)$. Further 
$f$ belongs to $L^1(\R)$. We can therefore apply Fubini's Theorem and we obtain that 
\[
\frac{1}{2\pi}\,\int_{0}^{\infty}\widehat{f}(t)
\widehat{wv}(t)T_t\, dt =
\frac{1}{2\pi}\, \int_{-\infty}^{\infty} 
f(s)\Bigr(\int_{0}^{\infty}\widehat{wv}(t)
e^{-its}T_t\, dt\Bigl)\,ds.
\]
Note that for any $s\in\R$, 
$$
\int_{0}^{\infty}\widehat{wv}(t)
e^{-its}T_t\, dt = \Gamma(A+is,\widehat{wv}).
$$
According to the multiplicativity of the 
Hille-Phillips functional calculus, we have
\[
 \int_{0}^{\infty}\widehat{wv}(t)e^{-its}T_t \,dt = 
\frac{1}{2\pi} 
\Big(\int_{0}^{\infty}\widehat{w}(r)e^{-irs}T_r\,dr 
\Big)\Big(\int_{0}^{\infty}\widehat{v}(t)e^{-its}T_{t}\, dt\Big). 
\]
Let $W,V\colon\R\to B(H)$ be defined by
$$
W(s) = \int_{0}^{\infty} \widehat{w}(r)e^{-irs}T_r\,dr
\qquad\hbox{and}\qquad
V(s) = \int_{0}^{\infty} \widehat{v}(t)e^{-its}T_t\,dt,
\qquad s\in\R.
$$
It follows from above that for any $x,x^*\in H$, we have
\begin{equation}\label{equaforgammbound}
\bigl\langle \Gamma\bigl(A, (2\pi)^{-1}\widehat{f}\widehat{h}\bigr)x,x^* \bigr\rangle = 
\frac{1}{4\pi^2}\int_{-\infty}^\infty f(s) \langle 
W(s) x, V(s)^*x^* \rangle\,ds .
\end{equation}
Applying the Cauchy-Schwarz inequality, we deduce
\[
\big|\big\langle \Gamma\bigl(A, (2\pi)^{-1}
\widehat{f}\widehat{h}\bigr)x,x^* \big\rangle\big| 
\leq \frac{1}{4\pi^2}\norme{f}_{\infty}\Bigl(\int_{-\infty}^\infty\norme{W(s)x}^2ds \Bigr)^{\frac{1}{2}}
\Bigl(\int_{-\infty}^\infty \norme{V(s)^*x^*}^2 ds\Bigr)^{\frac{1}{2}}. 
\]
According to the Fourier-Plancherel equality on $L^2(\R;H)$, we have
$$
\int_{-\infty}^\infty \norme{W(s)x}^2 ds 
= 2\pi \int_{0}^\infty |\widehat{w}(r)|^2\norme{T_r x}^2dr.
$$
This implies
$$
\int_{-\infty}^\infty \norme{W(s)x}^2 ds 
\leq 2\pi C^2\int_{0}^\infty |\widehat{w}(r)|^2\norme{x}^2 dr \,
=  4\pi^2 C^2 \norme{w}^2_2\norme{x}^2.
$$
Similarly, we have
\[
\int_{-\infty}^\infty \norme{V(s)^*x^*}^2ds \leq 4\pi^2 C^2\norme{v}^2_2\norme{x^*}^2.
\]
Hence,
$$
\big|\big\langle \Gamma\bigl(A, (2\pi)^{-1}
\widehat{f}\widehat{h}\bigr)x,x^* \big\rangle\big|  
\leq  C^2\norme{f}_{\infty}\norme{w}_2\norme{v}_2 \norm{x}\norm{x^*}.
$$
Since this is true for all $x,x^*$, we have proved that 
$$
\bignorm{\Gamma\bigl(A, (2\pi)^{-1}\widehat{f}\widehat{h}\bigr)}
\leq  C^2\norme{f}_{\infty}\norme{w}_2\norme{v}_2 .
$$

Now let $h$ be an arbitrary element of 
$H^1(\R)$. As is well-known (see e.g. 
\cite[Exercise 1, p. 84]{gar}), there exist $w,v\in H^2(\R)$ 
such that $h =wv$
and $\norme{w}_2^2 =\norme{v}_2^2 = \norme{h}_1$.

Since $\mathcal{F}(\mathcal{S}(\R)) = \mathcal{S}(\R)$,
it follows from (\ref{H2=FL2}) that 
$\mathcal{F}(H^2(\R)\cap\mathcal{S}(\R)) 
= L^2(\R_+) \cap \mathcal{S}(\R)$. Since 
$ L^2(\R_+) \cap \mathcal{S}(\R)$ is dense in $L^2(\R_+)$, we readily
deduce that $H^2(\R)\cap\mathcal{S}(\R)$ is dense in 
$H^2(\R)$.

Thus, there exist sequences $(w_k)_{k\in \N}$, $(v_k)_{k\in \N} $ 
in $H^2(\R)\cap\mathcal{S}(\R)$ such that 
$w_k \to w$ and  $v_k \to v$ in $H^2(\R)$, when $k\to\infty$. 
This implies that 
$$
\norm{w_k}_2\norm{v_k}_2\longrightarrow\norm{h}_1,
$$
when $k\to\infty$ and
$w_kv_k \to wv=h$ in $H^1(\R)$, when $k\to\infty$.
Consequently,
$$
\bignorm{\Gamma\bigl(A, (2\pi)^{-1}\widehat{f}\widehat{w_kv_k}\bigr)
-\Gamma\bigl(A, (2\pi)^{-1}\widehat{f}\widehat{h}\bigr)} \longrightarrow 0
$$
when $k\to\infty$. Indeed,
\begin{align*}
\bignorm{\Gamma\bigl(A, (2\pi)^{-1}\widehat{f}\widehat{w_kv_k}\bigr)
-\Gamma\bigl(A, (2\pi)^{-1}\widehat{f}\widehat{h}\bigr)} 
&= \norme{\int_{0}^{\infty} \widehat{f}(t)\widehat{(w_kv_k-wv)}(t)T_t\,dt} \\
& \leq C\norm{\widehat{f}}_1\norm{\widehat{(w_kv_k-wv)}}_\infty \\
& \leq C\norm{\widehat{f}}_1\norm{w_kv_k-wv}_1.
\end{align*}
For all $k \in \N$, we have
$$
\bignorm{\Gamma\bigl(A, (2\pi)^{-1}\widehat{f}\widehat{w_kv_k}\bigr)} \leq  
C^2\normeinf{f}\norme{u_k}_2\norm{v_k}_2,
$$
by the first part of the proof. Passing to the limit, we obtain (\ref{inelemcalfon}).
\end{proof}

We now arrive at the main result of this subsection.

\begin{thm}\label{main1}
There exists a unique bounded homomorphism $\rho_{0,A}\colon
\A_0(\C_+)\to B(H)$ such that
\begin{equation}\label{main2}
\rho_{0,A}(L_b)=\int_0^\infty b(t)T_t\,dt
\end{equation}
for all $b\in L^1(\R_+)$. Moreover $\norm{\rho_{0,A}}\leq C^2$.
\end{thm}

\begin{proof}
By Lemma \ref{key} and the density of $C_{00}(\R)$ in $C_0(\R)$,
there exists a unique bounded bilinear map 
\[
u_A \colon C_0(\R)\times H^1(\R) \longrightarrow B(H)
\]
such that $u_A(f,h) = 
\Gamma\bigl(A,(2\pi)^{-1}\widehat{f}\widehat{h}\bigr)$ 
for each $(f,h)\in C_{00}(\R) \times H^1(\R)$. 
Moreover $\norm{u_A}\leq C^2$.

For each $\varepsilon >0$, $(A+\varepsilon)$ is 
the negative generator of the semigroup 
$(e^{-\varepsilon t}T_t)_{t \geq 0}$. Therefore, 
in the same manner as above, one can define
$u_{A+\varepsilon} \colon C_0(\R)\times H^1(\R) \rightarrow B(H)$ 
and we have the uniform estimate
\begin{equation}\label{estuAeps}
\forall \varepsilon > 0, \qquad 
\bignorm{u_{A+\varepsilon} : C_0(\R)\times H^1(\R) 
\longrightarrow B(H)} \leq  C^2.
\end{equation}

We claim that for each $\varepsilon>0$, we have 
\begin{equation}\label{proofpropcalfoncstep1}
u_{A+\varepsilon}(f,h)=
(f\star h)^\sim(A+\varepsilon), \qquad f\in C_0(\R),\, h\in H^1(\R). 
\end{equation}	
(We recall that the operator on the right-hand side
is defined by the half-plane holomorphic functional
calculus. In particular the above formula
shows that $(f\star h)^\sim(A+\varepsilon)$ is bounded.)
Recall that if $f\in C_{00}(\R)$, then $b=(2\pi)^{-1}\widehat{f}\widehat{h}
\in L^1(\R_+)$ and $(f\star h)^\sim=L_b$. Hence 
(\ref{proofpropcalfoncstep1}) is given by
Lemma \ref{lemcompatibility} in this case. In the general case, let 
$(f_n)_{n\in \N}$ be a sequence of $C_{00}(\R)$ such that 
$f_n \to f$ in $C_0(\R)$, when $n\to\infty$. Then
$u_{A + \varepsilon}(f_n,h) \to u_{A + \varepsilon}(f,h)$,
hence $(f_n\star h)^\sim(A+\varepsilon)
\to u_{A + \varepsilon}(f,h)$.
Moreover 
$(f_n\star h)^\sim\to (f\star h)^\sim$ in $H^\infty(\C_+)$. 
Therefore by the Convergence Lemma \ref{CVLemma}, 
$(f\star h)^\sim(A+\varepsilon)$ is bounded and 
\eqref{proofpropcalfoncstep1} holds true.

Next we show that in $B(H)$ we have
\begin{equation}\label{eps-zero}
u_{A+\varepsilon}(f,h)
\underset{\varepsilon \rightarrow 0}\longrightarrow 
u_A(f,h),\qquad f\in C_0(\R),\, h\in H^1(\R).
\end{equation}
In the case when $f\in C_{00}(\R)$,
$$
u_{A+\varepsilon}(f,h)=
\,\frac{1}{2\pi}\,\int_0^\infty
\widehat{f}(t)\widehat{h}(t)e^{-\varepsilon t} T_t\,dt
$$
for all $\varepsilon\geq 0$. Hence 
$$
\bignorm{u_{A}(f,h) -u_{A+\varepsilon}(f,h)}
\leq\,\frac{C}{2\pi}
\int_0^\infty\bigl\vert
\widehat{f}(t)\widehat{h}(t)\bigr\vert (1-e^{-\varepsilon t})\, dt.
$$
This integral tends to $0$ when $\varepsilon\to 0$,
by Lebesgue's dominated convergence theorem. This
yields the result in this case.
The general case follows from the density of $C_{00}(\R)$
in $C_0(\R)$ and the uniform estimate (\ref{estuAeps}).

We now construct $\rho_{0,A}$.
Let $F\in \A_0(\R)$ and consider
two sequences $(f_k)_{k\in \N}$ of $C_0(\R)$ and $(h_k)_{k\in \N}$ of $H^1(\R)$ satisfying \eqref{ineA} and \eqref{equA2}. We let
$$
F_N=\sum_{k=1}^N f_k\star h_k,\qquad N\geq 1.
$$
For any fixed $\varepsilon>0$,
it follows from 
\eqref{proofpropcalfoncstep1} that for any $N\geq 1$,
$$
\widetilde{F_N}(A+\varepsilon)
= \sum_{k=1}^N u_{A+\varepsilon}(f_k,h_k).
$$
We have both that $\widetilde{F_N}
\to \widetilde{F}$ in $H^\infty(\C_+)$ and that
$\sum_{k=1}^N u_{A+\varepsilon}(f_k,h_k)\to 
\sum_{k=1}^\infty u_{A+\varepsilon}(f_k,h_k)$ in $B(H)$.
Appealing again to Lemma \ref{CVLemma}, we deduce that 
$\widetilde{F}(A+\varepsilon)\in B(H)$ and that 
\begin{equation}\label{proofpropcalfoncstep2}
\widetilde{F}(A+\varepsilon) = \sum_{k=1}^\infty
u_{A+\varepsilon}(f_k,h_k).
\end{equation}

We observe that 
\begin{equation}\label{proofpropcalfoncstep3}
\sum_{k=1}^\infty
u_{A+\varepsilon}(f_k,h_k)  
\underset{\varepsilon \rightarrow 0}
\longrightarrow \sum_{k=1}^\infty
u_A(f_k,h_k)
\end{equation}
in $B(H)$.
To check this, let $a>0$ and choose $N\geq 1$ such that $\sum_{k=N+1}^\infty
\norm{f_k}_\infty\norm{h_k}_1\leq a$. 
We  have
\begin{align*}
\Bignorm{
\sum_{k=1}^{\infty} 
u_{A+\varepsilon} (f_k,h_k)\,-\,
\sum_{k=1}^{\infty} 
u_{A} (f_k,h_k)} &\,\leq 
\Bignorm{\sum_{k=1}^{N} 
u_{A+\varepsilon} (f_k,h_k)\,-\,
\sum_{k=1}^{N} 
u_{A} (f_k,h_k)}\\ & +\,
\sum_{k=N+1}^{\infty} \norm{u_{A+\varepsilon}(f_k,h_k)}\ +\,
\sum_{k=N+1}^{\infty} \norm{u_{A}(f_k,h_k)}.
\end{align*}
By the uniform estimate (\ref{estuAeps}),
this implies that
$$
\Bignorm{
\sum_{k=1}^{\infty} 
u_{A+\varepsilon} (f_k,h_k)\,-\,
\sum_{k=1}^{\infty} 
u_{A} (f_k,h_k)} \,\leq 
\Bignorm{\sum_{k=1}^{N} 
u_{A+\varepsilon} (f_k,h_k)\,-\,
\sum_{k=1}^{N} 
u_{A} (f_k,h_k)} +\,2Ca.
$$
Applying (\ref{eps-zero}), we deduce that 
$$
\Bignorm{
\sum_{k=1}^{\infty} 
u_{A+\varepsilon} (f_k,h_k)\,-\,
\sum_{k=1}^{\infty} 
u_{A} (f_k,h_k)} \,\leq 3Ca
$$
for $\varepsilon>0$ small enough, which shows the result.

Combining (\ref{proofpropcalfoncstep2})
and (\ref{proofpropcalfoncstep3}) we obtain that 
$\widetilde{F}(A+\varepsilon)$ has a limit in $B(H)$, 
when $\varepsilon\to 0$. We set
$$
\rho_{0,A}(\widetilde{F}) :=\lim_{\varepsilon\to 0} \widetilde{F}
(A+\varepsilon).
$$
It is plain that $\rho_{0,A}\colon \A_0(\C_+)\to B(H)$
is a linear map. It follows from the construction that
$$
\norm{\rho_{0,A}(\widetilde{F})}_{\A_0}\leq C^2
\norm{\widetilde{F}}_{\A_0}
$$
for any
$F\in\A_0(\R)$, hence $\rho_{0,A}$ is bounded with $\norm{\rho_{0,A}}\leq C^2$.

Let $b\in L^1(\R_+)$. By the compatibility Lemma  
\ref{lemcompatibility}, we have
$$
L_b(A+\varepsilon) =\int_0^\infty  b(t)e^{-\varepsilon t}
 T_t\,dt
$$
for all $\varepsilon >0$. Passing to the limit and using
Lebesgue's dominated convergence theorem, we obtain
(\ref{main2}).

It follows from the density of $\{L_b \,:\, b\in L^1(\R_+)\}$ 
in $\A_0(\C_+)$, given by Lemma \ref{Laplace},
that $\rho_{0,A}$ is unique. 
Morever the multiplicativity of the Hille-Phillips functional
calculus ensures that $\rho_{0,A}$ is a Banach algebra homomorphism.
\end{proof}

\begin{rq1}\label{Indep1} Let $F\in \A_0(\R)$ and let
$(f_k)_{k\in \N}$ and $(h_k)_{k\in \N}$ be sequences 
of $C_0(\R)$ and $H^1(\R)$, respectively, satisfying \eqref{ineA} 
and \eqref{equA2}. It follows from the proof of Theorem \ref{main1}
that
\begin{equation}\label{Indep2}
\rho_{0,A}(\widetilde{F}) =\,\sum_{k=1}^\infty u_A(f_k,h_k).
\end{equation}
This equality shows that the right-hand side of (\ref{Indep2})
does not depend on the choice of $(f_k)_{k\in \N}$ and $(h_k)_{k\in \N}$.
The reason why we did not take (\ref{Indep2})
as a definition of $\rho_{0,A}$ is precisely 
that we did not know a priori
that $\sum_{k=1}^\infty u_A(f_k,h_k)$ was independent of the 
representation of $F$.
\end{rq1}

\subsection{Functional calculus on $\A(\C_+)$}\label{FCA}
We keep the notation from the
previous subsection.
We can extend Theorem \ref{main1} as follows.

\begin{cor}\label{main3}
There exists a unique bounded homomorphism $\rho_{A}\colon
\A(\C_+)\to B(H)$ extending
$\rho_{0,A}$. Moreover $\norm{\rho_A}\leq C^2$.
\end{cor}

\begin{proof} We follow an idea from \cite{bgt1}, using regularization.
Consider the sequence $(G_N)_{N\in\N}$ defined
in the proof of Proposition \ref{normequA0}. Then 
$$
\widetilde{G_N}(z)=\,\frac{N}{N+z}\,,\qquad z\in\C_+,\, N\geq 1.
$$

For any $\varphi\in\A(\C_+)$, we let $S_\varphi$
be the operator defined by
$$
S_\varphi=(1+A)\rho_{0,A}(\varphi\widetilde{G_1}),
$$
with domain ${\rm Dom}(S_\varphi)=\{x\in H\, :\, 
[\rho_{0,A}(\varphi \widetilde{G_1})](x)\in {\rm Dom}(A)\}$. In this definition,
we use the fact that 
$\varphi\widetilde{G_1}$ belongs to $\A_0(\C_+)$, which
follows from Proposition \ref{propAlg}. It is clear that 
$S_\varphi$ is closed.
Further ${\rm Dom}(A)\subset {\rm Dom}(S_\varphi)$, hence $S_\varphi$
is densely defined. More precisely,
if $x\in {\rm Dom}(A)$, then $x=
\rho_{A,0}(\widetilde{G_1})(1+A)(x)$
hence 
$$
[\rho_{0,A}(\varphi \widetilde{G_1})](x) = \rho_{0,A}(\varphi \widetilde{G_1}^2)](1+A)(x)
=(1+A)^{-1}\rho_{0,A}(\varphi \widetilde{G_1})](1+A)(x)
$$
belongs to ${\rm Dom}(A)$ and we have
\begin{equation}\label{DomA}
S_\varphi(x)= \rho_{0,A}(\varphi\widetilde{G_1})(1+A)(x).
\end{equation}

Since $\rho_{0,A}$ is multiplicative, we have
$\rho_{0,A}(\varphi \widetilde{G_N}\widetilde{G_1})=
\rho_{0,A}(\varphi \widetilde{G_N})(1+A)^{-1}$
for any $N\geq 1$. 
Moreover as noticed in the proof of Proposition \ref{normequA0},
$(G_N)_{N\in\N}$ is an approximate unit of $\A_0(\R)$,
hence $\varphi \widetilde{G_N}\widetilde{G_1}\to 
\varphi\widetilde{G_1}$ in $\A_0(\C_+)$, when $N\to \infty$.
We deduce, using (\ref{DomA}), that for any $x\in{\rm Dom}(A)$,
$$
S_\varphi(x) 
=\lim_N \rho_{0,A}(\varphi \widetilde{G_N}\widetilde{G_1})(1+A)(x)
= \lim_N \rho_{0,A}(\varphi \widetilde{G_N})(x).
$$
For any $N\geq 1$, 
$$
\norm{\rho_{0,A}(\varphi \widetilde{G_N})}
\leq C^2 \norm{\varphi\widetilde{G_N}}_{\A_0}
\leq C^2 \norm{\varphi}_{\A},
$$
by (\ref{ideal}). Consequently, $\norm{S_\varphi(x)}
\leq C^2 \norm{\varphi}_{\A}\norm{x}$ for any 
$x\in{\rm Dom}(A)$. This shows that
${\rm Dom}(S_\varphi)=H$ and $S_\varphi\in B(H)$.

We now define $\rho_{A}\colon
\A(\C_+)\to B(H)$ by $\rho_A(\varphi)=S_\varphi$. It is clear
from above that $\rho_{A}$ is linear and
bounded, with $\norm{\rho_A}\leq C^2$. It extends $\rho_{0,A}$
because if $F\in\A_0(\C_+)$, then we have
$\rho_{0,A}(\varphi\widetilde{G_1}) = 
\rho_{0,A}(\widetilde{G_1})\rho_{0,A}(\varphi)=(1+A)^{-1}
\rho_{0,A}(\varphi)$, hence $S_\varphi=\rho_{0,A}(\varphi)$.

Let $\varphi_1,\varphi_2\in \A(\C_+)$. For any integers $N_1,N_2\geq 1$,
we have 
$$
\rho_{0,A}\bigl(\varphi_1\varphi_2 \widetilde{G_{N_1}}\widetilde{G_{N_2}}\bigr)
= \rho_{0,A}(\varphi_1 \widetilde{G_{N_1}})
\rho_{0,A}(\varphi_2 \widetilde{G_{N_2}}),
$$
because $\rho_{0,A}$ is multiplicative. We deduce that
$\rho_{0,A}(\varphi_1\varphi_2 \widetilde{G_{N_1}})
= \rho_{0,A}(\varphi_1 \widetilde{G_{N_1}})
\rho_{A}(\varphi_2)$ for all $N_1\geq 1$, by
letting $N_2\to\infty$. Next  we obtain  
$\rho_{A}(\varphi_1\varphi_2)
= \rho_{A}(\varphi_1)
\rho_{A}(\varphi_2)$ by letting 
$N_1\to\infty$. Thus $\rho_{A}$ is multiplicative.

The uniqueness property is clear.
\end{proof}

We will show in Remark \ref{narrow} below that the functional
calculus $\rho_A$ from Corollary \ref{main3} is compatible with the 
Hille-Phillips functional calculus on $M(\R_+)$.

\subsection{Operators with a bounded $H^\infty(\C_+)$-functional
calculus}\label{H-infty}

The goal of this subsection is to explain 
the connections between our main results
(Theorem \ref{main1}, Corollary \ref{main3})
and $H^\infty$-functional calculus.

We will
assume that the reader is familiar with sectorial
operators and their $H^\infty$-functional calculus, 
for which we refer 
to \cite{haa1} or \cite[Chapter 10]{hnvw}. Using
standard notation, for any 
$\theta\in(0,\pi)$ we let $\Sigma_\theta=\{z\in\C^*\, :\,
\vert{\rm Arg}(z)\vert<\theta\}$ and 
$$
H_0^\infty(\Sigma_{\theta})
=\bigl\{\varphi 
\in H^{\infty}(\Sigma_{\theta})
\,:\,\exists s>0,\, \vert \varphi(z)\vert \lesssim \min\{
{|z|^s,\vert z\vert^{-s}}\} \text{ on } \Sigma_{\theta}\bigr\}.
$$

Let $(T_t)_{t\geq 0}$ be a bounded $C_0$-semigroup 
on some Banach space 
$X$, with generator $-A$. Recall that 
$A$ is a sectorial operator of type $\frac{\pi}{2}$.

The following lemma is probably known to specialists, 
we include a proof
for the sake of completeness. In part (i),
the operator $\varphi(A)$ is defined 
by (\ref{fofA}) whereas in part (ii),
the operator $\varphi(A)$ is defined 
by \cite[(2.5)]{haa1}. It is worth noting that
if $\varphi\in {\mathcal E}({\mathcal H}_\alpha)\cap
H_0^\infty(\Sigma_{\theta})$, then these two definitions
coincide.

\begin{lem}\label{Calcul-H}
The following assertions are equivalent.
\begin{itemize} 
\item [(i)] There exists a constant $C>0$ such that
for all $\alpha<0$ and for all $\varphi\in 
{\mathcal E}({\mathcal H}_\alpha)$,
\begin{equation}\label{Calcul-HH}
\norm{\varphi(A)}\leq C\norm{\varphi}_{H^\infty(\C_+)}.
\end{equation}
\item [(ii)] There exists a constant $C>0$ such that
for all $\theta\in\bigl(\frac{\pi}{2},\pi\bigr)$ and for all 
$\varphi\in H_0^\infty(\Sigma_{\theta})$,
$$
\norm{\varphi(A)}\leq C\norm{\varphi}_{H^\infty(\C_+)}.
$$
\item [(iii)] There exists a constant $C>0$ such that
for all $b\in L^1(\R_+)$,
\begin{equation}\label{Calcul-HHH}
\Bignorm{\int_{0}^\infty b(t)T_t\, dt}\leq 
C\norm{\widehat{b}}_\infty.
\end{equation}
\end{itemize}
\end{lem}

\begin{proof}
Assume (i). By the approximation argument at the beginning
of \cite[Section 5]{bat-haa}, (\ref{Calcul-HH}) holds as
well for any $\varphi\in H^\infty({\mathcal H}_\alpha)$.
Let $b\in L^1(\R_+)$. The function 
$L_b(\,\cdotp+\varepsilon)$ belongs
to $H^\infty({\mathcal H}_{-\varepsilon})$ for any $\varepsilon>0$,
hence we have
$$
\norm{L_b(A+\varepsilon)}\leq C
\norm{L_b(\,\cdotp+\varepsilon)}_{H^\infty(\C_+)}
\leq  C
\norm{L_b}_{H^\infty(\C_+)} = C\norm{\widehat{b}}_\infty.
$$
Applying Lemma \ref{lemcompatibility} and letting 
$\varepsilon\to 0$,
we obtain (\ref{Calcul-HHH}), which proves (iii).

The fact that (iii) implies (ii) follows from 
\cite[Lemma 3.3.1 $\&$ Proposition 3.3.2]{haa1},
see also \cite[Lemma 2.12]{lem2}.

Assume (ii) and let us prove (i). 
For any $\varepsilon\in (0,1)$,
consider the rational function $q_\varepsilon$ defined by 
$$
q_\varepsilon(z)
= \frac{\varepsilon +z}{1+\varepsilon z},
\qquad z\not= \frac{-1}{\varepsilon}.
$$
We may and do assume that 
$\alpha\in (-1,0)$ when proving (i). 
Fix some  $\varepsilon\in (0,1)$.
It is easy to check (left to the reader)
that
$q_\varepsilon$ maps ${\mathcal H}_\alpha$ into
itself. Moreover there exists 
$\theta\in\bigl(\frac{\pi}{2},\pi\bigr)$ such that 
$q_\varepsilon$ maps $\Sigma_{\theta}$ into
$\C_+$. 

Let $\varphi\in {\mathcal E}({\mathcal H}_\alpha)$, then
$$
\varphi_\varepsilon : = \varphi\circ q_\varepsilon\colon
{\mathcal H}_\alpha\cup \Sigma_\theta
\longrightarrow\C
$$
is a well-defined bounded holomorphic function.
Moreover we have 
\begin{equation}\label{unifo}
\norm{\varphi_\varepsilon}_{H^{\infty}({\mathcal H}_\alpha)}
\leq \norm{\varphi}_{H^{\infty}({\mathcal H}_\alpha)}.
\end{equation}
By \cite[Lemma 2.2.3]{haa1},
$\varphi_\varepsilon$ belongs 
to  
$H_0^\infty(\Sigma_{\theta})\oplus
{\rm Span}\{1,(1+\,\cdotp)^{-1}\}$. Further
the definition of $\varphi_\varepsilon(A)$ provided
by the functional calculus of sectorial operators
coincides with 
the definition of $\varphi_\varepsilon(A)$ provided by the 
half-plane functional calculus. Hence for some constant
$C'>0$ not depending on $\varepsilon$, we have 
$$
\norm{\varphi_\varepsilon(A)}\leq C'\norm{\varphi_\varepsilon}_{H^\infty(\C_+)}
\leq C'\norm{\varphi}_{H^\infty(\C_+)},
$$
by (ii). Since $\varphi_\varepsilon\to \varphi$ pointwise on ${\mathcal H}_\alpha$,
it now follows from (\ref{unifo}) and the Convergence Lemma \ref{CVLemma} 
that $\norm{\varphi(A)}\leq C'\norm{\varphi}_{H^\infty(\C_+)}$,
which proves (i).
\end{proof}

We say that 
$A$ admits a bounded $H^\infty(\C_+)$-functional
calculus if one of (equivalently, all of)
the properties of Lemma \ref{Calcul-H} hold true.
If $A$ is sectorial of type $<\frac{\pi}{2}$,
the latter is equivalent to $A$ having 
a bounded $H^\infty$-functional
calculus of angle $\frac{\pi}{2}$ is the usual sense.
The main feature of the ``bounded
$H^\infty(\C_+)$-functional
calculus" property considered here is that it 
may apply to the case when the sectorial
type of $A$ is not  $<\frac{\pi}{2}$.

We now come back to the specific case when $X=H$ is a Hilbert space.
Here are a few known facts in this setting:
\begin{itemize}
\item [(f1)] If $(T_t)_{t\geq 0}$ is a contractive semigroup
(that is, $\norm{T_t}\leq 1$ for all $t\geq 0$), then 
$A$ admits a bounded $H^\infty(\C_+)$-functional
calculus. See \cite[Section 7.1.3]{haa1} for a proof and
more on this theme.
\item [(f2)]  We say that
$(T_t)_{t\geq 0}$ is similar to a contractive semigroup if there
exists an invertible operator $S\in B(H)$ such that 
$(ST_tS^{-1})_{t\geq 0}$ is  a contractive semigroup.
A straightforward application of the previous result is that
in this case, $A$ admits a bounded $H^\infty(\C_+)$-functional
calculus.
\item [(f3)] If $A$ is sectorial of type $<\frac{\pi}{2}$, then
$A$ admits a bounded $H^\infty(\C_+)$-functional
calculus (if and) only if $(T_t)_{t\geq 0}$ is similar to a 
contractive semigroup. This goes back to \cite[Section 4]{lemsim}.
\item [(f4)] There exist sectorial operators of type $<\frac{\pi}{2}$
which do not admit a bounded $H^\infty(\C_+)$-functional
calculus, by \cite{mcya, baicle} 
(see also \cite[Section 7.3.4]{haa1}).
\item [(f5)] There exists a bounded $C_0$-semigroup
$(T_t)_{t\geq 0}$ such that $A$ admits a bounded $H^\infty(\C_+)$-functional
calculus but $(T_t)_{t\geq 0}$ is not 
similar to a contractive semigroup. This follows from
\cite[Proposition 4.8]{lemsim} and its proof.
\end{itemize}

\smallskip
We now establish analogues of Theorem \ref{main1}
and Corollary \ref{main3} in the case when 
$A$ admits a bounded $H^\infty(\C_+)$-functional
calculus. 
Just as we did in Subsection \ref{HP-Versions}, we set
$$
{\mathcal C}_0(\C_+)=\bigl\{\widetilde{F}\, :\, 
F\in C_0(\R)\cap H^\infty(\R)\bigr\}
\qquad\hbox{and}\qquad
{\mathcal C}(\C_+)=\bigl\{\widetilde{F}\, :\, 
F\in C_b(\R)\cap H^\infty(\R)\bigr\}.
$$
Since  $\{\widehat{b}(-\,\cdotp)\, :\, b\in L^1(\R_+)\}$
is dense in $C_0(\R)\cap H^\infty(\R)$, by Remark \ref{densities},
the following
is straightforward.

\begin{prop}\label{H-infini2}
Assume that $A$ admits a bounded $H^\infty(\C_+)$-functional
calculus on $H$. Then 
there exists a unique bounded homomorphism $\nu_{0,A}\colon
{\mathcal C}_0(\C_+) \to B(H)$ such that
\begin{equation}\label{main4}
\nu_{0,A}(L_b)=\int_0^\infty b(t)T_t\,dt
\end{equation}
for all $b\in L^1(\R_+)$. 
\end{prop}

Now arguing as in the proof of Corollary \ref{main3},
we deduce the following.

\begin{cor}\label{H-infini3}
Assume that $A$ admits a bounded $H^\infty(\C_+)$-functional
calculus on $H$. Then 
there exists a unique bounded homomorphism $\nu_{A}\colon
{\mathcal C}(\C_+) \to B(H)$ such that
(\ref{main4}) holds true for all
$b\in L^1(\R_+)$. 
\end{cor}

Of course when the above corollary applies, $\nu_{A}$
is an extension of the mapping $\rho_A$ from
Corollary \ref{main3}. Thus our main results
(Theorem \ref{main1}, Corollary \ref{main3})
should be regarded as a way to obtain a
``good" functional calculus for negative
generators of bounded $C_0$-semigroups 
which do not admit a bounded $H^\infty(\C_+)$-functional
calculus.

\subsection{Note added in May 2022}\label{Added}
A first version of this paper has circulated since  the beginning of 2021. 
A few months later, together with Safoura Zadeh, we proved in \cite[Section 4]{ALZ}
that the inclusion (\ref{MH-inclusion}) is actually an equality. Equivalently (see 
Theorem \ref{DualA0} above), we have 
$$
{\mathcal A}_0(\R)^*\simeq \MH.
$$
The paper \cite{ALZ} also contains a new proof of Theorem \ref{main1} based on 
a description of the so-called $S^1$-bounded Fourier multipliers on
$H^1(\R)$ and on a tensor product estimate of independent interest, inspired by an 
old result of White \cite[Section 5]{white}.

\section{Comparison with the Besov functional calculus}\label{Besov}
In this section we compare the functional calculus constructed
in Section \ref{CF} (Theorem \ref{main1} and Corollary \ref{main3})
with the Besov functional calculus from \cite[Subsection 5.5]{haa3} and 
\cite{bgt1}.
We start with some background on the analytic homogeneous Besov space
used in the latter paper. We refer to \cite[Section 6]{bgt1} for further details.

Let $\psi\in \SR$ such that ${\rm Supp}(\psi)\subset \bigl[\frac12,2\bigr]$,
$\psi(t)\geq 0$ for all $t\in\R$, and
$\psi(t)+\psi(\frac{t}{2}\bigr)=1$ for all $t\in 
\bigl[1,2\bigr]$. For any $k\in\Z$, we let 
$\psi_k\in\SR$ be defined by $\psi_k(t)=\psi(2^{-k}t)$, $t\in\R$.
A key property of the sequence $(\psi_k)_{k\in\tiny\Z}$
is that for any $k_0\in\Z$, we have
\begin{equation}\label{sum1}
\forall\, t\in [2^{k_0},2^{k_0+1}):\qquad
\psi_{k_0}(t)+\psi_{k_0+1}(t)=1\quad\hbox{and}\quad 
\psi_k(t)=0\ \hbox{if}\ k\notin\{k_0,k_0+1\}.
\end{equation}
Next  define $\phi_k={\mathcal F}^{-1}(\psi_k)$. It is plain that 
for any $k\in\Z$,
\begin{equation}\label{phi-k}
\phi_k\in H^1(\R)
\qquad\hbox{and}\qquad \norm{\phi_k}_1 = \norm{\phi_0}_1.
\end{equation} 
It follows that for any $F\in BUC(\R)$ and any $k\in\Z$, $F\star \phi_k$
belongs to $BUC(\R)\cap H^\infty(\R)$. We define a Besov
space $\B_0(\R)$ by
$$
\B_0(\R)=\biggl\{F\in BUC(\R)\, :\,
\sum_{k\in\tiny{\Z}}\norm{F\star \phi_k}_\infty\,<\infty
\ \hbox{ and }\ F=\sum_{k\in\tiny{\Z}}
F\star \phi_k \biggr\}.
$$
This is a Banach space for the norm
$$
\norm{F}_{\B_0} = \sum_{k\in\tiny{\Z}}\norm{F\star \phi_k}_\infty.
$$
This space is denoted by $\B_{dyad}$ in  \cite[Section 6]{bgt1}.

Next we set $\B_{00}(\R)= \B(\R)\cap C_0(\R)$, equipped with the norm 
of $\B_0(\R)$. Then $\B_{00}(\R)$ is a closed subspace
of $\B(\R)$ and we clearly have
$$
\B_{00}(\R)\subset C_0(\R)\cap H^\infty(\R)
\qquad\hbox{and}\qquad 
\B_0(\R)\subset BUC(\R)\cap H^\infty(\R).
$$

We wish to underline that the above definitions
of $\B_0(\R)$ and $\B_{00}(\R)$ 
do not depend on the choice 
of the function $\psi$. More precisely if $\psi^{(1)},\psi^{(2)}$
are two functions as above and if we let 
$\B_0^{\psi^{(1)}}(\R)$ and $\B_0^{\psi^{(2)}}(\R)$
denote the associated spaces, then 
$\B_0^{\psi^{(1)}}(\R)$ and $\B_0^{\psi^{(2)}}(\R)$
coincide as vector spaces and the 
norms $\norm{\,\cdotp}_{\B_0^{\psi^{(1)}}}$ and 
$\norm{\,\cdotp}_{\B_0^{\psi^{(2)}}}$ are equivalent.
We refer to \cite[Section 6]{bgt1} and the references therein for 
these properties.

Similarly to Subsection \ref{HP-Versions}, we 
introduce half-plane versions of $\B_0(\R)$ and $\B_{00}(\R)$,
by setting
$$
\B_{00}(\C_+) = \bigl\{\widetilde{F}\, :\,
F\in \B_0(\R)\bigr\}
\qquad\hbox{and}\qquad 
\B_0(\C_+) = \bigl\{\widetilde{F}\, :\,
F\in \B(\R)\bigr\}.
$$
According to \cite[Proposition 6.2]{bgt1}, the space
$\B_0(\C_+)\subset H^\infty(\C_+)$ coincides 
with the space $\B_0$ considered by Batty-Gomilko-Tomilov
in \cite[Subsection 2.2]{bgt1}. By \cite[Subsection 2.4]{bgt1}, we have
\begin{equation}\label{LM}
\bigl\{L_b\, :\, b\in L^1(\R_+)\bigr\}\,\subset\, 
\B_{00}(\C_+).
\end{equation}

Moreover Batty-Gomilko-Tomilov established the following remarkable functional
calculus result.

\begin{thm}\label{BGT} (\cite[Theorem 4.4]{bgt1}, 
\cite[Theorem 6.1]{bgt2})
Let $X$ be a Banach space, let $(T_t)_{t\geq 0}$
be a bounded $C_0$-semigroup on $X$ and let $-A$ 
denote its generator.
The following are equivalent.
\begin{itemize}
\item [(i)] There exists a constant $K>0$ such that
$$
\int_{-\infty}^{\infty}\bigl\vert
\langle R(\beta +it, A)^2(x),x^*\rangle\bigr\vert\, dt
\,\leq \,\frac{-K}{\beta}\,\norm{x}\norm{x^*}
$$
for all $\beta<0$, all $x\in X$ and all $x^*\in X^*$.
\item [(ii)] There exists a 
bounded homomorphism $\gamma_A\colon \B_0(\C_+)\to B(X)$
such that
$$
\gamma_A(L_b) = \int_{0}^{\infty} b(t) T_t\, dt,\qquad b\in L^1(\R_+).
$$
\end{itemize}
In this case, $\gamma_A$ is unique.
\end{thm}

Condition (i) in Theorem \ref{BGT} goes back at least 
to \cite{gom} and \cite{shifeng}. In fact, condition (i)
can be defined for any closed and densely defined
operator $A$ satisfying (\ref{HP}) for
$\omega=0$. Then it follows from
\cite{gom,shifeng} that (i) actually implies
that $-A$ generates a bounded $C_0$-semigroup on $X$.
(See also \cite[Theorem 6.4]{bat-haa}.)
Conversely, if $X=H$ is a Hilbert space,
it is proved in \cite{gom,shifeng} that if
$-A$ generates a bounded $C_0$-semigroup, 
then $A$ satisfies (i). (The assumption that 
$X=H$ is a Hilbert space is crucial here, see
the beginning of Section \ref{Banach} for more on this.)

Thus if $(T_t)_{t\geq 0}$
is a bounded $C_0$-semigroup with generator $-A$
on Hilbert space, then the property (ii) in
Theorem \ref{BGT} holds true.
It is therefore natural to compare 
Corollary \ref{main3} with that property.
This is the aim of the rest of this section.

\begin{prop}\label{embeddingBesovA}
We have
\[
\B_0(\C_+)\subset \A(\C_+) 
\qquad\hbox{and}\qquad 	
\B_{00}(\C_+)\subset  \A_0(\C_+).
\]
Moreover there exists a constant $K>0$ such that
$\norm{\varphi}_{\A}\leq K\norm{\varphi}_{\B_0}$
for any $\varphi\in \B_0(\C_+)$.
\end{prop}

\begin{proof}
It follows from (\ref{sum1}) that 
\begin{equation}\label{sum2}
\psi_k=\psi_k(\psi_{k-1}+\psi_k+\psi_{k+1}),\qquad k\in\Z.
\end{equation}
Consequently,
$$
\phi_k=\phi_k\star\bigl(\phi_{k-1}+\phi_{k}+\phi_{k+1}\bigr), \qquad k\in\Z.
$$
Let $F\in \B(\R)$. Applying the above identity, we have
$$
F = \sum_{k\in\tiny\Z} F\star \phi_k \star
\bigl(\phi_{k-1}+\phi_{k}+\phi_{k+1}\bigr).
$$
Appealing to (\ref{phi-k}), we observe that
$F\star \phi_k\in BUC(\R)$ and  
$\phi_{k-1}+\phi_{k}+\phi_{k+1}\in H^1(\R)$ for each
$k\in\Z$, and that 
$$
\sum_{k\in\tiny\Z} \norm{F\star \phi_k}_\infty
\norm{\phi_{k-1}+\phi_{k}+\phi_{k+1}}_1\,\leq 3\norm{\phi_0}_1\norm{F}_{\B_0}.
$$
This shows that $F\in\A(\R)$, with 
$$
\norm{F}_\A\leq 3\norm{\phi_0}_1\norm{F}_{\B_0}.
$$
This yields $\B(\C_+)\subset  \A(\C_+)$.
The above argument also shows 
that $\B_0(\C_+)\subset  \A_0(\C_+)$.
\end{proof}

Let $H$ be a Hilbert space and let $A$ be the negative generator
of a bounded $C_0$-semigroup on $H$. 
We already noticed that $A$ satisfies property 
(ii) in Theorem \ref{BGT}. According to
Proposition \ref{embeddingBesovA} and (\ref{main2}), 
the functional calculus $\rho_A\colon\A(\C_+)\to B(H)$
from Corollary \ref{main3} extends the 
functional calculus $\gamma_A\colon\B(\C_+)\to B(H)$.

It turns out that the extension from $\gamma_A$ to
$\rho_A$ is an actual improvement, because of the
following result.

\begin{thm}\label{AdifferentB0}
We have
\[
\B_0(\C_+)\not= \A(\C_+) 
\qquad\hbox{and}\qquad 	
\B_{00}(\C_+)\not=  \A_0(\C_+).
\]
\end{thm}

We need some preparation before coming to the proof.
We use an idea from \cite[Paragraph 2.6.4]{triebel}.
First for the definition of the Besov space
$\B_0(\R)$, we make the additional assumption that 
$\psi(t)=1$ for any $t\in\bigl[\frac34,1\bigr]$.
This is allowed by the aforementioned fact that the definition
of $\B_0(\R)$ does not depend on $\psi$. This implies that 
${\rm Supp}(\psi)\subset \bigl[\frac12,\frac32]$.
Second we fix a non-zero
function
$f_0\in\SR$ such that ${\rm Supp}(f_0)\subset \bigl[\frac34,1\bigr]$.
Next for any integer $n\geq 0$, we set $N_n=2^n-1$ 
and $f_n = \tau_{N_n} f_0 = f_0(\,\cdotp -N_n)$. By construction,
${\rm Supp}(\psi_k)\subset \bigl[2^{k-1},\frac32 2^k]$
for all $k\in\Z$ and
${\rm Supp}(f_n)\subset \bigl[2^n-\frac14,2^n]$
for all $n\geq 0$. We derive that
\begin{equation}\label{product}
\forall\, k\geq 0 :\qquad
f_k\psi_k = f_k
\quad\hbox{and}\quad f_n\psi_k=0\ \hbox{if}\ n\not=k,
\end{equation}
as well as
\begin{equation}\label{product2}
\forall\, n,n'\geq 0 : \qquad f_n f_{n'}=0\ \hbox{if}\ n\not=n'.
\end{equation}

\begin{lem}\label{lemBnoeqA}
There exists a bounded continuous
function $m \colon \R_+^*\to\C$ such that 
\begin{equation}\label{eqB1infty}
\underset{k\in \Z}\sup 
\norm{\widehat{m\psi_k}}_1 < \infty 
\end{equation}
and the mapping $T_m\colon H^2(\R)\to H^2(\R)$ does not belong to
$\MH$.
\end{lem}

\begin{proof} Using the definitions preceding the lemma, we set
$$
m(t) =\sum_{n=0}^\infty e^{iN_n t}f_n(t),\qquad t>0.
$$
At most one term is non zero in this sum, hence this
is well-defined 
and $m\in C_b(\R_+^*)$. Let $k\geq 0$.
According to (\ref{product}), we have
$m\psi_k = e^{iN_k\,\cdotp} f_k$ hence
$$
\norm{\widehat{m\psi_k}}_1
= \norm{\widehat{f_k}(\cdotp - N_k)}_1
=\norm{\widehat{f_k}}_1 =\norm{\widehat{f_0}}_1.
$$
Since $m\psi_k=0$ if $k<0$, this
shows (\ref{eqB1infty}).

Define
$$
g_N = {\mathcal F}^{-1}\Bigl(
\sum_{n=0}^N 
e^{-iN_n\,\cdotp}f_n\Bigr)
$$
for all $N\geq 0$. 
Then $g_N\in \SR\cap H^1(\R)$ 
hence $g_N\in H^p(\R)$ for any $1\leq p\leq\infty$.
Let us estimate its $L^p$-norm.
On the one hand, we have
$$
\norm{g_N}_1\leq 
\sum_{n=0}^N \bignorm{{\mathcal F}^{-1}(e^{-iN_n\,\cdotp}f_n)}_1
= \sum_{n=0}^N \bignorm{\bigl[{\mathcal F}^{-1}(f_n)\bigr](\,\cdotp -N_n))}_1
= \sum_{n=0}^N \bignorm{{\mathcal F}^{-1}(f_n)}_1,
$$
hence
$$
\norm{g_N}_1\leq  (N+1)\norm{{\mathcal F}^{-1}(f_0)}_1.
$$
On the other hand, for any $t\in\R$,
we have
$$
g_N(t) = \sum_{n=0}^N 
{\mathcal F}^{-1}(f_n)(t-N_n),
$$
hence
$$
\vert g_N(t)\vert \leq \sum_{n=0}^N \vert
g_0(t-N_n)\vert.
$$
Since $g_0=\mathcal{F}^{-1}(f_0)\in\SR$ we infer that 
$$
\sup_{N\geq 0}\norm{g_N}_\infty\,<\infty.
$$

For any $1<p<\infty$, we have
$\norm{g_N}_p\leq\norm{g_N}_1^{\frac{1}{p}}
\norm{g_N}_\infty^{1-\frac{1}{p}}$, hence 
the above estimates imply the existence
of a constant $K>0$ such that
\begin{equation}\label{Np}
\norm{g_N}_p\leq K N^{\frac{1}{p}},\qquad N\geq 1.
\end{equation}

By (\ref{product2}), we have
$$
m\widehat{g_N} = 
\Bigl(\sum_{n=0}^\infty e^{iN_n\,\cdotp}f_n
\Bigr)\Bigl(
\sum_{n=0}^N 
e^{-iN_n\,\cdotp}f_n\Bigr) = \sum_{n=0}^N f_n^2.
$$
For any $n\geq 0$, ${\rm Supp}(f_n^2)\subset [2^{n-1},2^n]$
hence by \cite[Theorem 5.1.5.]{gra}, we have an estimate
$$
\bignorm{{\mathcal F}^{-1}(m\widehat{g_N})}_p\approx
\Bignorm{\Bigl(\sum_{n=0}^N \vert
{\mathcal F}^{-1}(f_n^2)\vert^2\Bigr)^\frac12}_p.
$$
Further, $f_n^2 = f_0^2(\,\cdotp -N_n)$ hence 
$\vert
{\mathcal F}^{-1}(f_n^2)\vert=\vert
{\mathcal F}^{-1}(f_0^2)\vert$ for any $n\geq 0$. Consequently,
$$
\Bigl(\sum_{n=0}^N \vert
{\mathcal F}^{-1}(f_n^2)\vert^2\Bigr)^\frac12
=(N+1)^\frac12\vert
{\mathcal F}^{-1}(f_0^2)\vert.
$$
Thus we have
$$
\bignorm{{\mathcal F}^{-1}(m\widehat{g_N})}_p\approx
N^\frac12.
$$
Comparing with (\ref{Np}) we deduce that if $2<p<\infty$, then
$T_m\colon H^2(\R)\to H^2(\R)$ is not a bounded
Fourier multiplier on $H^p(\R)$.
By Lemma \ref{noMpisnotM1}, we deduce that 
$T_m\notin \MH$. 
\end{proof}

\begin{proof}[Proof of Theorem \ref{AdifferentB0}]
If $\A(\C_+)$ were equal to $\B_0(\C_+)$, we would have
$\A_0(\C_+)=\B_{00}(\C_+)$ which in turn is equivalent to
$\A_0(\R) =\B_{00}(\R)$. So it suffices to show that
this equality fails. 
Let us assume, by contradiction, that
$\A_0(\R) =\B_{00}(\R)$.

Let $m$ be given by Lemma \ref{lemBnoeqA}. 
Let $b\in L^1(\R_+)$. For any $k\in\Z$, we have
$$
\psi_kb = \,\frac{1}{2\pi}\,{\mathcal F}\bigl(
\phi_k\star \widehat{b}(-\,\cdotp)\bigr).
$$
Hence using (\ref{sum1}), (\ref{sum2})
and Lemma \ref{Tool}, we have
\begin{align*}
\int_{-\infty}^{\infty} m(t)b(t)\,dt 
& = \sum_{k\in\Z} 
\int_{-\infty}^{\infty}
m(t)\psi_k(t)b(t)\,dt\\
&= \sum_{k\in \Z} \int_{-\infty}^{\infty}
\bigl(\psi_{k-1}(t)+\psi_k(t)+\psi_{k+1}(t)\bigr) m(t)\psi_k(t)b(t)\,dt \\
&=\,\frac{1}{2\pi}\,\sum_{k\in \Z} \int_{-\infty}^\infty 
\bigl[\mathcal{F}\bigl((\psi_{k-1}+\psi_k+\psi_{k+1})m\bigr)\bigr](u)
\bigl[\mathcal{F}(\psi_k b)\bigr](-u)\, du \\
&= \,\sum_{k\in \Z} 
\int_{-\infty}^\infty 
\bigl[\mathcal{F}\bigl((\psi_{k-1}+\psi_k+\psi_{k+1})m\bigr)\bigr](u)
\bigl[\phi_k\star\widehat{b}(-\,\cdotp)\bigr](u)\, du
\end{align*}
Therefore,
$$
\Bigl\vert 
\int_{-\infty}^\infty
m(t)b(t)\,dt\,\Bigr\vert\,\leq\,
\sum_{k\in \Z} 
\bignorm{\mathcal{F}\bigl((\psi_{k-1}+\psi_k+\psi_{k+1})m\bigr)}_1
\bignorm{\phi_k\star\widehat{b}(-\,\cdotp)}_\infty.
$$
Applying  \eqref{eqB1infty}, we deduce the existence of a constant 
$K>0$ such that
$$
\Bigl\vert 
\int_{-\infty}^\infty
m(t)b(t)\,dt\,\Bigr\vert\,\leq\,
K \sum_{k\in \Z} \bignorm{\phi_k\star\widehat{b}(-\,\cdotp)}_\infty
= K\norm{\widehat{b}(-\,\cdotp)}_{\B_0}.
$$
Therefore there exists $\eta\in\B_{00}(\R)^*$ such that
$$
\langle\eta,\widehat{b}(-\,\cdotp)\rangle
= 
\int_{-\infty}^\infty
m(t)b(t)\,dt,\qquad b\in L^1(\R_+).
$$
By assumption, $\eta\in\A_0(\R)^*$. Applying 
Theorem \ref{DualA0}, let $T\in\MH$ be associated to $\eta$ and let 
$m_0\in C_b(\R_+^*)$ be the symbol of $T$. Then by Remark 
\ref{DualA0+}, we have
$$
\langle\eta,\widehat{b}(-\,\cdotp)\rangle
= 
\int_{-\infty}^\infty
m_0(t)b(t)\,dt,\qquad b\in L^1(\R_+).
$$
We deduce that $m_0=m$, and this
contradicts the fact that $T_m\notin\MH$.
\end{proof}

We conclude this section with a series of remarks.

\begin{rq1}\label{narrow} 
Let $A$ be as in Subsections \ref{MAIN} and \ref{FCA}.
Let $\mu\in M(\R_+)$, with $\mu(\{0\})=0$. 
According to \cite[Subsection 2.2 $\&$ Proposition 6.2]{bgt1},
its Laplace transform $L_\mu\colon\C_+\to\C$ belongs
to $\B_0(\C_+)$. Hence $L_\mu$ belongs
to $\A(\C_+)$, by Proposition \ref{embeddingBesovA}.
The argument in the proof of Corollary
\ref{main3} shows that $\rho_{A}(L_\mu)$ is the strong limit
of $\rho_{0,A}(L_\mu\widetilde{G_N})$, when $N\to\infty$.
Define $c_N(t)=
Ne^{-Nt}$ for any $t> 0$ and recall that 
$\widetilde{G_N} = L_{c_N}$. Then $L_\mu \widetilde{G_N} = L_{\mu\star c_N}$
for any $N\geq 1$. Further $\mu\star c_N\to \mu$ narrowly, when $N\to\infty$.
It therefore follows from (\ref{main2}) that
$$
[\rho_A(L_\mu)](x) = \,\int_{\R_+} T_t(x)\,d\mu(t),\qquad x\in H.
$$
\end{rq1}

\begin{rq1}\label{D}  
Let ${\mathcal D}\subset H^1(\R)$ be the space of all $h\in H^1(\R)$
such that ${\rm Supp}(\widehat{h})$ is a compact subset of $\R_+^*$.
It is well-known that ${\mathcal D}$ is dense in
$H^1(\R)$. To check this, take any 
$h\in H^1(\R)$ and recall that there exist
$v,w \in H^2(\R)$ such that $h=wv$. Let 
$(d_n)_{n\in\N}$ and $(c_n)_{n\in\N}$ be sequences
of $C_b(\R_+^*)$ with compact supports
such that $d_n\to \widehat{w}$ and
$c_n\to \widehat{v}$ in $L^2(\R_+)$. Then 
$\mathcal{F}^{-1}(d_n)\to w$ and
$\mathcal{F}^{-1}(c_n)\to v$ in $H^2(\R)$, hence
$\mathcal{F}^{-1}(d_n)\mathcal{F}^{-1}(c_n)\to h$ in $H^1(\R)$.
Now it is easy to see that $\mathcal{F}^{-1}(d_n)\mathcal{F}^{-1}(c_n)$
belongs to ${\mathcal D}$ for any $n\in\N$.

Let $BUC\star {\mathcal D}\subset \A(\R)$ be the linear span of the functions
$f\star h$, for $f\in BUC(\R)$ and $h\in {\mathcal D}$. 
It follows from above that 
this is a dense subspace of $\A(\R)$.

Let ${\mathcal G}\subset H^\infty(\R)$ be the space of all $F\in H^\infty(\R)$
such that ${\rm Supp}(\widehat{F})$ 
is a compact subset of $\R_+^*$. Then we have
$$
BUC\star {\mathcal D}\subset {\mathcal G}\subset \B_0(\R).
$$
The first inclusion is obvious and the second one 
is given by \cite[Lemma 2.4]{bgt1}.

It follows that $\B_0(\R)$ is dense in $\A(\R)$, or equivalently that 
$\B_0(\C_+)$ is dense in $\A(\C_+)$.

Also $\B_{00}(\C_+)$ is dense in $\A_{00}(\C_+)$, by (\ref{LM})
and Lemma \ref{Laplace}.
\end{rq1}

\begin{rq1}\label{nonunital}
It follows from \cite[Subsection 2.2]{bgt1} 
that for any $\varphi\in\B_0(\C_+)$,
$\lim_{y\to\infty} \varphi(y)=0$, where the limit is taken 
for $y$ going to $\infty$ along the real axis.
We noticed in Remark \ref{D} that $\B_0(\C_+)$ is dense in $\A(\C_+)$.
Since $\norm{\,\cdotp}_{H^\infty(\C_+)}\leq
\norm{\,\cdotp}_{\A(\C_+)}$,
this implies that any element of $\A(\C_+)$ is the uniform 
limit of a sequence of $\B_0(\C_+)$.
Consequently, $\lim_{y\to\infty} \varphi(y)=0$ for any
$\varphi\in\A(\C_+)$.

Thus the algebra $\A(\C_+)$ (equivalently, the algebra
$\A(\R)$) does not contain any non-zero constant function
and hence is not unital.
\end{rq1}

\begin{rq1}\label{open} The following problem is open: Let 
$-A$ be the generator of a bounded $C_0$-semigroup on a Hilbert space.
Is the Cayley transform $V=(A-I_H)(A+I_H)^{-1}$ power bounded?
This problem is discussed in \cite[Section 5.5]{bgt1}, to
which we refer for information.

Let $v\in H^\infty(\C_+)$ be defined by $v(z)=(z-1)(z+1)^{-1}$.
It follows from Remark \ref{Rational} that
for any integer $n\geq 1$, the function $\varphi_n\colon\C_+\to\C$
defined by
$\varphi_n(z) = v(z)^n - (-1)^n$
belongs to $\A_0(\C_+)$ and 
$$
V^n= (-1)^n I_H +\rho_{0,A}(\varphi_n).
$$
Hence $\norme{V^n} = O\bigl(\norme{\varphi_n}_{\A_0}\bigr)$.

Therefore it would be interesting to determine the behaviour
of $\norme{\varphi_n}_{\A_0}$. It is shown in \cite[Section 5.1]{bgt2}
that $\norme{\varphi_n}_{\B_0}\asymp {\rm log(n)}$. We do not know
if the asymptotic behaviour of
$\norme{\varphi_n}_{\A_0}$ differs from the one of $\norme{\varphi_n}_{\B_0}$.
\end{rq1}

\section{$\gamma$-Bounded semigroups on Banach spaces}\label{Banach}
In general, Theorem \ref{main1} and Corollary \ref{main3}
do not hold true if $H$ is replaced by an arbitrary Banach space.
Indeed it is shown in \cite[Corollary 6.7]{bgt2} that 
the translation
semigroup $(T_t)_{t\geq 0}$
on $L^p(\R)$, for $1\leq p\not=2<\infty$, does not satisfy
condition (i) in Theorem \ref{BGT}. Hence
by the latter theorem and Proposition \ref{embeddingBesovA},
the mapping 
$$
L_b\mapsto \int_{0}^{\infty} b(t)T_t\, dt\,,\qquad b\in L^1(\R_+),
$$
is not bounded with respect to the $\A_0(\C_+)$-norm.

In this section we will however establish  Banach space versions
of Theorem \ref{main1} and Corollary \ref{main3} on Banach spaces,
involving $\gamma$-boundedness. We start with some background 
and basic facts on this topic
and refer to \cite[Chapter 9]{hnvw} for details and more information.

Let $X$ be a Banach space. Let $(\gamma_n)_{n \geq 1}$ be a 
sequence of independent complex valued standard Gaussian variables 
on some probability space $\Sigma$ and let 
$G_0\subset L^2(\Sigma)$ be the linear span of the $\gamma_n$.
We denote by $G(X)$ the closure of 
\[
G_0\otimes X = \Bigl\{
\sum_{k=1}^N \gamma_k \otimes x_k\, :\, x_k \in X, \, N \in \N 
\Bigr\}
\]
in the Bochner space $L^2(\Sigma;X)$,
equipped with the induced norm. Next we let 
$G'(X^*)$ denote the closure of $G_0\otimes X^*$
in the dual space $G(X)^*$.

A bounded set $\mathcal{T}\subset B(X)$ is called 
$\gamma$-bounded if
there exists a constant $C \geq 0 $ 
such that for all finite sequences 
$(S_k)_{k=1}^{N} \subset \mathcal{T}$ and $(x_k)_{k=1}^{N} \subset X$, we have:
\begin{equation}\label{Rboundedness}
\Bignorm{\sum_{k=1}^N \gamma_k\otimes S_k(x_k)}_{G(X)} 
\leq C\Bignorm{\sum_{k=1}^N \gamma_k\otimes x_k}_{G(X)}.
\end{equation}
The least admissible constant $C$
in the above inequality is called the 
$\gamma$-bound of $\mathcal{T}$ and is denoted by 
$\gamma(\mathcal{T})$.  

Let $Z$ be any Banach space and let ${\rm Ball}(Z)$ denote
its closed unit ball. A bounded operator $\rho\colon Z\to B(X)$
is called $\gamma$-bounded if
the set $\rho({\rm Ball}(Z))\subset B(X)$ 
is $\gamma$-bounded. In this case
we set $\gamma(\rho) = \gamma(\rho({\rm Ball}(Z)))$.

We now turn to the definition of $\gamma$-spaces, which 
goes back to the paper \cite{kal-wei1} (which began to circulate
20 years ago).
Let $H$ be a Hilbert space. A bounded operator $T 
\colon H \rightarrow X$ is called $\gamma$-summing if 
\[
\norme{T}_{\gamma} := \sup\Bigl\{ 
\Bignorm{\sum_{k=1}^{N}\gamma_k\otimes T(e_k)}_{G(X)}\Bigr\} < \infty,
\]
where the supremum is taken over all finite 
orthonormal systems $(e_k)_{k=1}^{N}$ in $H$. 
We let $\gamma_{\infty}(H;X)$ denote the space of all 
$\gamma$-summing operators and we endow it with the norm 
$\norme{\,\cdotp}_{\gamma}$. Then $\gamma_{\infty}(H;X)$  is a Banach space. 
Any finite rank bounded operator 
is $\gamma$-summing. We let 
$$
\gamma(H;X)\subset
\gamma_{\infty}(H;X)
$$
denote the
closure of the space of  finite rank bounded operators 
in $\gamma_{\infty}(H;X)$. In the sequel, finite rank bounded operators
are represented by the algebraic tensor product $H^*\otimes X$
in the usual way.

Following \cite[Section 5]{kal-wei1}, we let $\gamma'_+(H^*;X^*)$
be the space of all bounded operators $S\colon H^*\to X^*$
such that 
$$
\norme{S}_{\gamma'}  := \sup\bigl\{ \vert {\rm tr}(T^*S)\vert\, \big\vert
\, T\colon H\to X,\, {\rm rank}(T)<\infty,\,\norm{T}_\gamma\leq 1\bigr\}\,<\infty.
$$
Then $\norme{\,\cdotp}_{\gamma'}$ is a norm on $\gamma'_+(H^*;X^*)$
and according to \cite[Proposition 5.1]{kal-wei1}, we have
\begin{equation}\label{Dual-gamma}
\gamma'_+(H^*;X^*)\,=\,\gamma(H;X)^*
\end{equation}
isometrically, through the duality pairing 
$$
(S,T)\mapsto {\rm tr}(T^*S),\qquad T\in 
\gamma(H;X),\ S\in \gamma'_+(H^*;X^*).
$$

We will focus on the
case when $H$ is an $L^2$-space. 
Let $(\Omega,\mu)$ be a $\sigma$-finite
measure space. We identify $L^2(\Omega)^*$
and $L^2(\Omega)$ in the usual way.
A function 
$\xi\colon\Omega \rightarrow X$ is called
weakly-$L^2$ if for each $x^* \in X^*$,  
the function $\langle x^*, \xi(\,\cdotp) \rangle$  
belongs to $L^2(\Omega)$. Then the operator
$x^*\mapsto \langle x^*, \xi(\,\cdotp) \rangle$ from
$X^*$ into $L^2(\Omega)$ is bounded.
If $\xi$ is both measurable and weakly-$L^2$, then
its adjoint takes values in $X$ and we let 
$\mathbb{I}_\xi \colon L^2(\Omega)\to X$ denote the resulting
operator. More explicitly, 
$$
\langle x^*, \mathbb{I}_\xi (g)\rangle=
\int_\Omega 
g(t)\langle x^*, \xi(t) \rangle\,d\mu(t)\, , 
\qquad g\in L^2(\Omega),\, x^*\in X^*.
$$ 
We let $\gamma(\Omega;X)$ be the space of all measurable and 
weakly-$L^2$ functions $\xi\colon\Omega \rightarrow X$ such that 
$\mathbb{I}_\xi$ belongs to $\gamma(L^2(\Omega);X)$, and
we write  
$\norme{\xi}_{\gamma} = \norme{\mathbb{I}_\xi}_{\gamma}$
for any such function.

Likewise a function 
$\zeta\colon\Omega \rightarrow X^*$ is called
weakly$^*$-$L^2$ if for each $x \in X$,  
the function $\langle \zeta(\,\cdotp),x \rangle$  
belongs to $L^2(\Omega)$. In this case,
the operator $x\mapsto\langle \zeta(\,\cdotp),x \rangle$
from $X$ into $L^2(\Omega)$ is bounded and we
let $\mathbb{I}_\zeta \colon L^2(\Omega)\to X^*$
denote its adjoint. 
We let $\gamma'_+(\Omega;X^*)$ be the space of all  
weakly$^*$-$L^2$ functions $\zeta\colon\Omega \rightarrow X$ such that 
$\mathbb{I}_\zeta$ belongs to $\gamma'_+(L^2(\Omega);X)$, and
we write  
$\norme{\zeta}_{\gamma'} = \norme{\mathbb{I}_\zeta}_{\gamma'}$
for any such function. 

Note that our space
$\gamma'_+(\Omega;X^*)$ is a priori bigger than the one from
\cite[Definition 4.5]{kal-wei1},
where only measurable 
functions $\Omega\to X^*$ are considered.

\begin{lem}\label{Integral}
For any $\xi\in\gamma(\Omega;X)$ and any 
$\zeta\in \gamma'_+(\Omega;X^*)$,
the function $t\mapsto \langle\zeta(t),\xi(t)\rangle$
belongs to $L^1(\Omega)$ and in the duality (\ref{Dual-gamma}), 
we have
$$
\langle \mathbb{I}_\zeta,\mathbb{I}_\xi\rangle
=\,\int_\Omega \langle\zeta(t),\xi(t)\rangle\, d\mu(t).
$$
Moreover
$$
\int_\Omega \vert\langle\zeta(t),\xi(t)\rangle\vert
\,d\mu(t)\,\leq \norm{\xi}_\gamma\norm{\zeta}_{\gamma'}.
$$
\end{lem}

If we consider measurable 
functions $\zeta\colon \Omega\to X^*$ only, the above statement
is provided by \cite[Corollary 5.5]{kal-wei1}. 
The fact that this holds as well in the more general
setting of the present paper follows from the proof of 
\cite[Theorem 9.2.14]{hnvw}.

The main result of this section is the following.

\begin{thm}\label{main5}
Let $(T_t)_{t\geq 0}$ be a bounded $C_0$-semigroup on
$X$ and let $A$ be its negative generator.
The following assertions are
equivalent.
\begin{itemize}
\item [(i)] The semigroup
$(T_t)_{t\geq 0}$ is $\gamma$-bounded, that is, 
the set ${\mathcal T}_A = 
\{T_t\, :\, t\geq 0\}$ is $\gamma$-bounded;
\item [(ii)] There exists a $\gamma$-bounded 
homomorphism $\rho_{0,A}
\colon\A_0(\C_+)\to B(X)$ such that 
(\ref{main2}) holds true for all $b\in L^1(\R_+)$.
\end{itemize}
In this case, $\rho_{0,A}$ is unique 
and $\gamma({\mathcal T}_A)
\leq \gamma(\rho_{0,A}) \leq \gamma({\mathcal T}_A)^2$.

Further
there exists a unique bounded 
homomorphism $\rho_{A}\colon \A(\C_+)\to B(X)$ extending 
$\rho_{0,A}$, this homomorphism is $\gamma$-bounded and 
$\gamma(\rho_{A}) =\gamma(\rho_{0,A})$.
\end{thm}

A thorough look at the proofs of Theorem \ref{main1} and Corollary 
\ref{main3} reveals that in Subsections \ref{MAIN}
and \ref{FCA}, the Hilbertian structure was used only in 
Lemma \ref{key}. So without any surprise the main point in 
proving Theorem \ref{main5} is the following 
$\gamma$-bounded version of Lemma \ref{key}.

\begin{lem}\label{key2}
Let $(T_t)_{t\geq 0}$ be a $\gamma$-bounded $C_0$-semigroup on
$X$ and let $A$ be its negative generator. Let 
$C= \gamma({\mathcal T}_A)$. Then the set 
\begin{equation}\label{Theset}
\Bigl\{\Gamma\bigl(A,(2\pi)^{-1}\widehat{f}\widehat{wv}\bigr)\, :\,
f\in C_{00}(\R),\, w,v\in H^2(\R)\cap\S(\R),\,
\{\norm{f}_\infty,\norm{w}_2,\norm{v}_2\}\leq 1\Bigr\}
\end{equation}
is $\gamma$-bounded, with $\gamma$-bound $\leq C^2$.
\end{lem}

\begin{proof}
Let $N\in\N$ and let $f_1,\ldots, f_N\in C_{00}(\R)$,
$w_1,\ldots,w_N,v_1,\ldots,v_N \in H^2(\R)\cap\S(\R)$
such that $\norm{f_k}_\infty\leq 1, 
\norm{w_k}_2\leq 1$ and $\norm{v_k}_2\leq 1$
for any $k=1,\ldots, N$. We set 
$$
S_k= \Gamma\bigl(A,(2\pi)^{-1}\widehat{f_k}\widehat{w_kv_k}\bigr),
\qquad k=1,\ldots, N.
$$
Let $x_1,\ldots,x_N\in X$ and $x_1^*,\ldots,x_N^*\in X^*$.
Following the notation in the proof of Lemma \ref{key}, we 
define, for any $k=1,\ldots, N$, two strongly 
continuous functions $W_k,V_k\colon \R\to B(X)$ by
$$
W_k(s) = \int_{0}^{\infty} \widehat{w_k}(r)e^{-irs}T_r\,dr
\qquad\hbox{and}\qquad
V_k(s) = \int_{0}^{\infty} \widehat{v_k}(t)e^{-its}T_t\,dt,
\qquad s\in\R.
$$
According to (\ref{equaforgammbound}),
$$
\sum_{k=1}^{N} \langle S_k(x_k),x_k^*\rangle\,
=\,
\frac{1}{4\pi^2}\,\sum_{k=1}^{N}
\int_{-\infty}^\infty f_k(s) \langle 
W_k(s) x_k, V_k(s)^*x_k^* \rangle\,ds,
$$
hence
$$
\Bigl\vert \sum_{k=1}^{N} \langle S_k(x_k),x_k^*\rangle\Bigr\vert
\,\leq\,
\frac{1}{4\pi^2}\,\sum_{k=1}^{N}
\int_{-\infty}^\infty \bigl\vert
 \langle 
W_k(s) x_k, V_k(s)^*x_k^* \rangle
\bigr\vert\, ds.
$$

We let $\N_N=\{1,\ldots,N\}$ for convenience.
We will use $\gamma$-spaces on either $\R$ or $\N_N\times\R$.
For any $k=1,\ldots,N$, the function
$$
\alpha_k : = W_k(\,\cdotp)x_k\colon \R\longrightarrow X
$$
is measurable and
weakly-$L^2$. Likewise, 
$$
\beta_k : = V_k(\,\cdotp)^*x_k^*\colon \R\longrightarrow X^*
$$
is weakly$^*$-$L^2$.
If we are able to show that 
$\alpha_k\in \gamma(\R;X)$
and $\beta_k\in \gamma'_+(\R;X^*)$
for any $k=1,\ldots,N$, then  Lemma \ref{Integral} ensures that
\begin{equation}\label{4pi2}
\Bigl\vert \sum_{k=1}^{N} \langle S_k(x_k),x_k^*\rangle\Bigr\vert
\,\leq\,
\frac{1}{4\pi^2}\,
\bignorm{(k,s)\mapsto W_k(s)x_k}_{\gamma(\N_N\times\R;X)}
\bignorm{(k,s)\mapsto V_k^*(s)x_k^*}_{\gamma'(\N_N\times\R;X^*)}.
\end{equation}
Our aim is now to check that $\alpha_k\in \gamma(\R;X)$
and $\beta_k\in \gamma'_+(\R;X^*)$
for any $k$ and to estimate the right-hand side of 
(\ref{4pi2}).

By assumption, ${\mathcal T}_A=\{T_t\, :\, t\geq 0\}$ is $\gamma$-bounded.
According to the Multiplier Theorem stated as \cite[Theorem 6.1]{haa-roz},
there exists a bounded operator 
$$
M\colon \gamma(L^2(\R);X)\longrightarrow \gamma(L^2(\R);X)
$$
with norm $\leq C = \gamma({\mathcal T}_A)$,  
mapping $\gamma(\R;X)$ into itself, and such that 
for any $\xi\in\gamma(\R;X)$,
$[M(\xi)](t) = T_t(\xi(t))$ if $t\geq 0$,
and $[M(\xi)](t) = 0$ if $t< 0$.
Further by the Extension Theorem stated as \cite[Theorem 9.6.1]{hnvw},
${\mathcal F}\otimes I_X\colon L^2(\R)\otimes X\to L^2(\R)\otimes X$ admits a (necessarily unique)
bounded extension 
$$
\Psi\colon \gamma(L^2(\R);X)\longrightarrow \gamma(L^2(\R);X),
$$
with norm $\leq \sqrt{2\pi}$. 
According to  \cite[Lemma 2.19]{arn1},
$\mathbb{I}_{\alpha_k} = (\Psi\circ M)(\widehat{w_k}\otimes x_k)$
for any $k=1,\ldots,N$.
This shows that $\alpha_k\in \gamma(\R;X)$. 
Let $(e_k)_{k=1}^N$ be the canonical basis of 
$\ell^2_N$. It follows from above that
\begin{align*}
\bignorm{(k,s)\mapsto W_k(s)x_k}_{\gamma(\N_N\times\R;X)}
\, 
& = \Bignorm{\sum_{k=1}^N 
e_k\otimes (\Psi\circ M)(\widehat{w_k}\otimes x_k)}_{\gamma(L^2(\N_N\times\R);X)}
\\
& \leq \sqrt{2\pi}\, C\,
\Bignorm{\sum_{k=1}^N e_k \otimes\widehat{w_k}\otimes x_k}_{\gamma(L^2(\N_N\times\R);X)}.
\end{align*}
The finite sequence $(e_k \otimes\widehat{w_k})_{k=1}^N$ 
is an orthogonal
family of $L^2(\N_N\times\R)$. Consequently,
\begin{align*}
\Bignorm{\sum_{k=1}^N e_k \otimes\widehat{w_k}\otimes x_k}_{\gamma(L^2(\N_N\times\R);X)}\, &
=\Bignorm{\sum_{k=1}^N \norm{\widehat{w_k}}_2\gamma_k\otimes x_k}_{G(X)}\\
&\leq\max_k\norm{\widehat{w_k}}_2\,
\Bignorm{\sum_{k=1}^N  \gamma_k\otimes x_k}_{G(X)}.
\end{align*}
Since $\norm{\widehat{w_k}}_2=\sqrt{2\pi}\norm{w_k}_2\leq \sqrt{2\pi}$
for any $k=1,\ldots,N$, 
we finally obtain that
$$
\bignorm{(k,s)\mapsto W_k(s)x_k}_{\gamma(\N_N\times\R;X)}\,
\leq\,2\pi\, C\,\Bignorm{\sum_{k=1}^N 
\gamma_k\otimes x_k}_{G(X)}.
$$

We now analyse the $\beta_k$. Fix $k$
and consider $g\in L^2(\R)$ and $x\in X$.
Using Lemma \ref{Integral}, we have
\begin{align*}
\langle \mathbb{I}_{\beta_k}(g), x\rangle \,
& =\int_{-\infty}^\infty g(s)\langle
x_k^*, V_k(s)x\rangle\,ds\\
& =\int_{-\infty}^\infty
g(s) {\mathcal F}\bigl(
\widehat{v_k}\langle x_k^*, T_{\cdotp}(x)
\rangle\bigr)(s)\, ds\\
& =\int_{0}^\infty \widehat{g}(t) \widehat{v_k}(t) \langle x_k^*, 
T_t(x)\rangle\, dt\\
& =\bigl\langle 
\widehat{v_k}\otimes x_k^*, M(\widehat{g}\otimes x)\bigr\rangle\\
& =\bigl\langle 
\widehat{v_k}\otimes x_k^*, 
(M\circ\Psi)(g\otimes x)\bigr\rangle\\
& =\bigl\langle 
(\Psi^*\circ M^*)(\widehat{v_k}\otimes x_k^*),
g\otimes x\bigr\rangle.
\end{align*}
This shows that $\beta_k\in\gamma'_+(\R;X^*)$, 
with $\mathbb{I}_{\beta_k} = 
(\Psi^*\circ M^*)(\widehat{v_k}\otimes x_k^*)$.
Now arguing as in
the $W_k(\,\cdotp)x_k$ case, we obtain that
$$
\bignorm{(k,s)\mapsto V_k^*(s)x_k^*}_{\gamma'(\N_N\times\R;X^*)}
\,
\leq\,2\pi\, C\,\Bignorm{\sum_{k=1}^N 
\gamma_k\otimes x_k^*}_{G'(X)}.
$$

We now implement these estimates in (\ref{4pi2}) to obtain that
$$
\Bigl\vert \sum_{k=1}^{N} \langle S_k(x_k),x_k^*\rangle\Bigr\vert
\,\leq\,C^2 \Bignorm{\sum_{k=1}^N 
\gamma_k\otimes x_k}_{G(X)}\Bignorm{\sum_{k=1}^N 
\gamma_k\otimes x_k^*}_{G'(X)}.
$$
By the very definition of $G'(X)$, this means that 
$$
\Bignorm{\sum_{k=1}^N 
\gamma_k\otimes S_k(x_k)}_{G(X)}\,\leq\,
C^2 \Bignorm{\sum_{k=1}^N 
\gamma_k\otimes x_k}_{G(X)},
$$
which completes the proof.
\end{proof}

\begin{proof}[Proof of Theorem \ref{main5}]
Assume (i). By Lemma \ref{key2}, any element in the set
(\ref{Theset}) has norm $\leq C^2$. Hence 
the proof of Theorem \ref{main1} 
shows
the existence of a unique bounded 
homomorphism $\rho_{0,A}
\colon\A_0(\C_+)\to B(X)$ such that 
(\ref{main2}) holds true for all $b\in L^1(\R_+)$.

To prove $\gamma$-boundedness of $\rho_{0,A}$, we 
introduce the set
$$
{\mathcal L}\,=\,
\bigl\{(f\star wv)^{\sim}\, :\,
f\in C_{00}(\R),\, w,v\in H^2(\R)\cap\S(\R),\,
\{\norm{f}_\infty,\norm{w}_2,\norm{v}_2\}\leq 1\bigr\}\,
\subset\A_{0}(\C_+).
$$
Recall (see the proof of Lemma \ref{key}) that any $h\in H^1(\R)$ can be written
as a product $h=wv$, with 
$w,v\in H^2(\R)$ and $\norm{w}_2^2=\norm{v}_2^2=\norm{h}_1$,
and that $H^2(\R)\cap\S(\R)$ is dense in $H^2(\R)$.
Going back to Definition \ref{defNA}, we 
derive that 
$$
{\rm Ball}(\A_0(\C_+)) =\overline{\rm Conv}\{\mathcal L\}.
$$
This implies that
$$
\rho_{0,A}\bigl({\rm Ball}(\A_0(\C_+))\bigr)
\subset 
\overline{\rm Conv}\bigl\{\rho_{0,A}(\mathcal L)\bigr\}.
$$
Since $\rho_{0,A}((f\star wv)^{\sim}) =
\Gamma\bigl(A,(2\pi)^{-1}\widehat{f}\widehat{wv}\bigr)$
for any $f\in C_{00}(\R)$
and any 
$w,v\in H^2(\R)\cap\S(\R)$, Lemma \ref{key2}
says that $\rho_{0,A}(\mathcal L)$ is 
$\gamma$-bounded, with $\gamma$-bound $\leq C^2$. Owing to the 
fact that $\gamma$-boundedness 
and $\gamma$-bounds are preserved by 
convex hulls (see e.g. \cite[Proposition 8.1.21]{hnvw})
and 
uniform limits, we infer that 
$\rho_{0,A}$ is $\gamma$-bounded, with $\gamma(\rho_{0,A})\leq C^2$.
This proves (ii).

Conversely assume (ii).
The proof of Corollary \ref{main3} shows
the existence 
of a unique bounded 
homomorphism $\rho_{A}\colon \A(\C_+)\to B(X)$
extending $\rho_{0,A}$ as well as the 
fact that $\rho_{A}\bigl({\rm Ball}(\A(\C_+))\bigr)$
belongs to the strong closure of 
$\rho_{0,A}\bigl({\rm Ball}(\A_0(\C_+))\bigr)$.
Since $\gamma$-boundedness 
and $\gamma$-bounds are preserved by strong limits,
we obtain that 
$\rho_{A}$ is $\gamma$-bounded, with $\gamma(\rho_{A}) =\gamma(\rho_{0,A})$.

Finally the argument in Remark \ref{narrow} (1) shows that 
for any $t>0$, 
$$
T_t\in \rho_{A}\bigl({\rm Ball}(\A(\C_+))\bigr).
$$
This implies (i), with 
$\gamma({\mathcal T}_A)\leq \gamma(\rho_A)$.
\end{proof}

\vskip 0.5cm
\noindent
{\bf Acknowledgements.} The two authors were supported by the ANR project Noncommutative
analysis on groups and quantum groups (No./ANR-19-CE40-0002). The first author was also 
supported by the ERC grant Rigidity of groups and higher index theory under the European Union’s Horizon 2020 research and innovation program (grant
agreement no. 677120-INDEX). 

We gratefully thank the referee for providing a simplification of the proof
of Theorem \ref{thmmutlH1} which improved the original argument, as well as  for other useful remarks.

\vskip 0.5cm

\bibliographystyle{plain}

\bibliography{article}

\vskip 0.5cm

\end{document}